\documentclass[fontsize=11pt,toc=bib]{scrartcl}

\setkomafont{title}{\sffamily}
\setkomafont{pagehead}{\sffamily}
\setkomafont{pagenumber}{\sffamily}
\setkomafont{author}{\sffamily}
\addtokomafont{author}{\vspace*{2em}\large}

\setkomafont{dedication}{\footnotesize\raggedleft\itshape}
\usepackage[automark]{scrlayer-scrpage}
\ihead{\headmark}
\ohead{\pagemark}
\usepackage[
  paper=a4paper,
  left=80pt,
  right=80pt,
  top=86pt,
  bottom=56pt,
  bindingoffset=0mm,
  includefoot,
  includehead,
  foot=16.5pt,
  head=16.5pt
]{geometry}

\DeclareMathAlphabet{\pazocal}{OMS}{zplm}{m}{n}
\DeclareMathAlphabet{\mymathbb}{U}{bbold}{m}{n}
\DeclareMathAlphabet{\dutchcal}{U}{dutchcal}{b}{n}

\usepackage{bbm}
\usepackage{floatrow}
\usepackage{longtable}
\usepackage{pdflscape}
\usepackage{tcolorbox}
\usepackage{comment}
\usepackage{enumitem}
\usepackage{mathtools}
\usepackage{fontawesome}
\usepackage{booktabs}
\usepackage{indentfirst}
\usepackage{lipsum}
\usepackage[english]{babel}
\usepackage{amsmath} 
\usepackage{amsthm}
\usepackage{thmtools}
\usepackage[sfdefault]{FiraSans}
\usepackage[charter,expert]{mathdesign}
\usepackage[scaled=.96]{XCharter}
\usepackage{setspace}
\usepackage{nicefrac}
\usepackage[dvipsnames]{xcolor}
    \definecolor{linkcolor}{HTML}{0D6A9E}
    \definecolor{3blue}{HTML}{0072B2}
    \definecolor{3green}{HTML}{009E73}
    \definecolor{3ochre}{HTML}{E69F00}
    \definecolor{3yellow}{HTML}{F0E442}
    \definecolor{3cyan}{HTML}{56B4E9}
    \definecolor{3red}{HTML}{D55E00}
    \definecolor{3pink}{HTML}{CC79A7}
    \definecolor{2blue}{HTML}{1A85FF}
    \definecolor{2red}{HTML}{D41159}
\usepackage{graphicx}
    \graphicspath{ {./gfx/} }
\usepackage{caption}
\usepackage{tensor}
\usepackage{tikz-cd}
    \tikzcdset{arrow style=tikz}
    \tikzset{vertl/.style={anchor=south, rotate=90, inner sep=1mm}}
    \tikzset{vertr/.style={anchor=south, rotate=-90, inner sep=1mm}}
\usepackage[hang,flushmargin]{footmisc} 
\usepackage{hyperref}
    \hypersetup{
    colorlinks=true,
    linktoc=all,
    linkcolor=linkcolor,
    citecolor=linkcolor,
    filecolor=magenta,      
    urlcolor=linkcolor,
    pdfauthor=Davide Dal Martello,
    pdftitle=Generic Hecke algebras in the infinite
    }
\usepackage{cleveref}

\newtheoremstyle{komait}
      {\topsep}   
      {\topsep}   
      {\itshape}  
      {0pt}       
      {\bfseries\sffamily} 
      {.}         
      {5pt plus 1pt minus 1pt} 
      {}          
    \newtheoremstyle{komanormal}
      {\topsep}   
      {\topsep}   
      {\rmfamily}  
      {0pt}       
      {\bfseries\sffamily} 
      {.}         
      {5pt plus 1pt minus 1pt} 
      {}          
\theoremstyle{komait}
    \newtheorem{theorem}{Theorem}[section]
    
    \newtheorem{lemma}[theorem]{Lemma}
    \newtheorem{definition}[theorem]{Definition}
    \newtheorem{proposition}[theorem]{Proposition}
    \newtheorem{conjecture}[theorem]{Conjecture}
    \newtheorem*{braid*}{Braid theorem}
\theoremstyle{komanormal}
    \NewCommandCopy{\proofqedsymbol}{\qedsymbol}
    \newcommand{\openboxthickness}{0.4pt} 
        \renewcommand{\openbox}{\leavevmode
          \hbox to.77778em{%
            \hfil\vrule width \openboxthickness
            \vbox to.675em{%
              \hrule width \dimexpr.675em-2\dimexpr\openboxthickness height \openboxthickness
              \vfil
              \hrule height\openboxthickness
            }%
            \vrule width \openboxthickness
            \hfil
          }%
        }
    \renewcommand{\openboxthickness}{.675pt}
        \AtBeginEnvironment{proof}{\renewcommand{\qedsymbol}{\proofqedsymbol}}
    \newtheorem{example}[theorem]{Example}
        \AtBeginEnvironment{example}{\pushQED{\qed}\renewcommand{\qedsymbol}{$\triangle$}}
        \AtEndEnvironment{example}{\popQED\endexample}
    
        \AtBeginEnvironment{exercise}{\pushQED{\qed}\renewcommand{\qedsymbol}{$\triangle$}}
        \AtEndEnvironment{exercise}{\popQED\endexample}
    \newtheorem{remark}[theorem]{Remark}
        \AtBeginEnvironment{remark}{\pushQED{\qed}\renewcommand{\qedsymbol}{$\triangle$}}
        \AtEndEnvironment{remark}{\popQED\endexample}

\newcommand{\B}{\mathrm{Br}}
\newcommand{\A}{\mathrm{Ar}}
\newcommand{\Lin}{\mathrm{Lin}(W^\alpha_n)}
\newcommand{\Tran}{\mathrm{Tran}(W^\alpha_n)}
\newcommand{\tu}[1]{\tau_{\!#1}}
\newcommand{\sg}[1]{\sigma_{\!#1}}
\newcommand{\isg}[1]{\sigma_{\!#1}^{-1}}
\newcommand{\diag}{\mathrm{diag}}

\newcommand{\naturali}{\mathbb{N}}
\newcommand{\integer}{\mathbb{Z}}

\newcommand{\complex}{\mathbb{C}}

\begin{document}

\title{\usekomafont{subtitle}\LARGE\vspace{-2.5em}Generic Hecke algebras in the infinite}
\author{\raisebox{-.5ex}{\href{https://orcid.org/0000-0002-4975-8774}{\includegraphics[height=15pt]{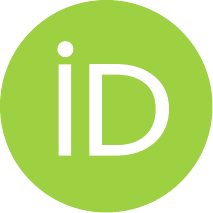}}}\hspace{.5em}Davide DAL MARTELLO\footnote{\hspace{.4em}Department of Mathematics, University of Padua, Via Trieste 63, 35121 Padova, IT\\\faEnvelopeO\hspace*{.5em}davide.dalmartello@unipd.it, \href{mailto:contact@davidedalmartello.com}{contact@davidedalmartello.com}}}
\date{\vspace{-2.5em}}


\maketitle

\begin{abstract}\noindent
Aiming for a revival of the theory of crystallographic complex reflection groups, we compute (minimal) Coxeter-like reflection presentations for the infinite families of those non-genuine groups which satisfy Steinberg's fixed point theorem. These new presentations behave à la Coxeter, encoding many of the group's properties at a glance, and their signature feature---named the $x$-relation---is fully understood in terms of configuration spaces. Crucially, the presentations further achieve the braid theorem, allowing to deform into the generic Hecke algebra. In particular, we revisit the affine GDAHA family in deformation terms of the most general class of Steinberg crystallographic complex reflection groups.
\end{abstract}

\noindent
\begin{center}{\vspace{-.5em}\small{\textbf{\textsf{Keywords}}\hspace{.5em}reflection groups, braid theorem, configuration spaces, generic Hecke algebras.}}
\end{center}

\renewcommand{\contentsname}{\textcolor{linkcolor}{Contents}}

\tableofcontents

\section{Introduction}

The initial motivation for this paper was to ground the theory of generalized double affine Hecke algebras (GDAHA, \Cref{def:GDAHA}) in geometric terms of reflections. GDAHAs arise---in the very spirit of the original DAHA theory---as quotients of group algebras for certain braid groups, are known to be formally flat, and expected to quantize Calogero-Moser spaces \cite{EGO2006}.

Unexpectedly, the hunt for the right class of reflection groups led us to unearth a niche of papers that pioneered the study of ``complex r-groups'' \cite{BS1978,KTY1982,Popov2022,TY1982}. The panorama we look at spans four types of (irreducible) reflection groups, separated by finiteness and choice of ground field between real and complex.
The beautiful classical theory of finite real reflection groups shows that, algebraically, they are exactly the finite Coxeter groups. In particular, they encompass all (finite) Weyl groups. On the one hand, failing finiteness one matches the affine Weyl groups of irreducible root systems. On the other hand, going complex one finds the well-known Shepard-Todd groups (see \cite{Cohen1976} for a modern unified take and \cite[§3]{PS2019} for a minimal primer). These groups have been studied extensively in the past decades: in particular, they all admit a reflection presentation delivering the \hyperref[braid-thm]{braid theorem}, whose Hecke deformation is algebraically flat---namely, free as a module over the ring of parameters (see \cite{Puente2017} for a round-up).

Thus, three out of four types of reflection groups are by now well-understood; on the contrary, the infinite complex setting remains, comparatively, underdeveloped and much less known.
With the present paper, we hope to ignite a revival of this foundational corner of algebra which, following in DAHA's footsteps, promises a wealth of additions to representation theory, special functions, integrable systems \cite{EOR2007}, singularity theory \cite{GK2010}, and beyond.

To start with, a complete classification originally developed in the eighties is now widely available \cite{Popov2022}: it allows the restriction to \emph{crystallographic} infinite complex reflection groups, describing their theory as the complex twin of the affine Weyl one. Indeed, the structure of any such reflection group $W$, acting on a $n$-dimensional $\complex$-vector space, is an extension of a rank $2n$ lattice $\mathrm{Tran}(W)$ by a Shepard-Todd reflection group $G=\mathrm{Lin}(W)$, and is almost always split. Whether $G$ is a complexification of an irreducible finite Weyl group determines the \emph{genuineness} of $W$ \cite{Malle1996}, and the non-genuine case has moduli with complex parameter $\alpha$ on the ``modular strip'' \cite{Popov2022}. In Popov's bracket notation labeling the $W$-invariant lattices (when multiple) by an integral subscript $k$, we write $[G]_k$ for a genuine group and $[G]_k^\alpha$ otherwise.

As we aim for the Hecke deformation, the present paper takes a topological approach to the subject. Extending the generic Hecke algebra for Shepard-Todd groups \cite{BMR1998} to the infinite setting requires the understanding of the group's \emph{regular orbit space} $N_W/W$, where $N_W$ denotes the space of regular elements.
Specifically, one needs a monodromy presentation for the fundamental group $\pi_1(N_W/W)$.
Under finiteness, Steinberg's regularity theorem ensures that $N_W=M_W$, where $M_W$ denotes the space $W$ acts upon after removal of all reflection hyperplanes, leading to the usual definition of the \emph{braid group} $\B(W):=\pi_1(M_W/W)$ that is equally valid for affine Weyl groups---as they also satisfy $N_W=M_W$.
It was first observed by Puente \cite{Puente2017} that this structural property can fail in the infinite complex setting, and mostly does for non-genuine groups, prompting the following
\begin{definition}
    A reflection group $W$, acting on a vector space $V$, is said to be \emph{Steinberg} if any vector in V, which is fixed by a non-identity group element of $W$, lies on a reflection hyperplane for $W$.
\end{definition}
In other words, the regular points under the action of a Steinberg reflection group are precisely those lying off of the reflection hyperplanes.
Going ``dense'' by focusing on infinite families of reflection groups, indexed by a natural number $n=\mathrm{dim}(V)\geq1$, in the present paper we thus restrict to the following cases:
\begin{theorem}[\cite{PS2019}]
    Let $W_n$ be an infinite family of Steinberg crystallographic complex reflection groups. When genuine, 
    \begin{equation}\label{genuine}
        W_n\in\big\{\  [G(3,1,n)]_1;\ [G(4,1,n)]_{1,2},\ [G(4,2,n)]_{1,2};\ [G(6,1,n)],\ [G(6,2,n)],\ [G(6,3,n)]\ \big\}.
    \end{equation}
    When non-genuine, there are just two possibilities:
    \begin{equation}\label{non-genuine}
        W_n^\alpha\in\big\{\ [G(1,1,n)]^\alpha;\ [G(2,1,n)]^\alpha_1\ \big\}.
    \end{equation}    
\end{theorem}
\begin{remark}
    Lifting the Steinberg assumption, $N_W\neq M_W$ and one should study the fundamental group $\pi_1(N_W/W)$ directly. Under genuineness, the two spaces are claimed to identify topologically, namely $\pi_1(N_W/W)\simeq\pi_1(M_W/W)$ \cite[Theorem 3.10]{Puente2017}, while an attempted proof is still missing in the non-genuine case.  
\end{remark}
Free from ambiguity in the definition of the braid group, we hunt for the needed monodromy presentation with a purely algebraic approach based on the braid theorem. The statement, namely the existence of a reflection presentation that, after removing the finite order relations for all generating reflections, reduces to a presentation for the braid group, is a beautiful match between algebra and topology. After the pioneering work of Artin, we denote the \emph{abstract} group obtained by such removal as $\A(W)$, the \emph{Artin group} of $W$. Thus, we can boil down the theorem to a single neat formula: 
\begin{braid*}
     There exists a reflection presentation of $W$ such that
     \begin{equation}\label{braid-thm}
         \A(W) \ \simeq \ \B(W).
     \end{equation}
\end{braid*}
\begin{remark}
Let us stress the essential difference between the two objects in \eqref{braid-thm} substantiating the theorem's appeal: the left hand side is abstract, purely algebraic, and (crucially) expressed as a simple presentation; the right hand side is canonical, truly topological, and (in principle) quite challenging to harness in symbolic terms.
\end{remark}
Of course, the theorem hinges entirely on the existence of such reflection presentation and, as of yet, no a priori recipe to define it is known. As a basic obstruction, the complex setting still lacks a canonical method for constructing a generating system of reflections. In contrast with the real settings, these deficiencies keep limiting the theory to be case-by-case driven. Nevertheless, combining individual results of Malle with Puente's extended presentation for $[G(6,3,n)]$, we can state the braid theorem for each and every genuine group of our interest:
\begin{theorem}[\cite{Malle1996,Puente2017}]\label{thm:genuine-pres}
    Let $W_n$ be an infinite family of genuine Steinberg crystallographic complex reflection groups. Then,
    \begin{enumerate}[label=\alph*.]
        \item \Cref{tab:genuine} gives $W_n$ a reflection presentation;
        \item Such reflection presentation achieves the braid theorem.
    \end{enumerate}
\end{theorem}
Let us highlight the main features of these presentations. 
Up to an extra order relation, they are encapsulated by diagrams à la Coxeter whose nodes are labeled by the order of the corresponding generating reflection, whenever non-involutive.
As in the affine Weyl case, each diagram differs from the one codifying its Shepard-Todd part $\mathrm{Lin}(W_n)$ \cite[Tables 1 and 2]{BMR1998} by a single additional node---with the only exception of $[G(4,2,2)]_1$, which requires an additional pair.

Therefore, by forgetting order relations (including the extra one) and labeling we obtain a presentation for the respective braid group.
It is then easy to notice that genuine $[G(d,1,n)]_1$ share the regular orbit space for all possible $d=3,4,6$.
Back to our initial motivation, the triple of GDAHAs in types $E^{(1)}_{6,7,8}$ arises as deformation of the very same topological space.
This is no coincidence: the two spaces identify via an isomorphism \eqref{E-typeiso} which, crucially, survives the Hecke deformation. Being our target the affine type, only one GDAHA is then left to fully achieve our original goal.
Actually, the definition for the missing type $D^{(1)}_4$ gives nothing but the celebrated DAHA of Sahi \cite{Sahi1999}, itself recovered in the language of elliptic Hecke algebras (EHA) for type $C_n^{(1,1)}$ \cite{SS2009}. In order to give the algebra one more description à la Hecke, this time based on complex reflection groups, we must drop genuineness.

Our first result extends \Cref{thm:genuine-pres} (a) to the non-genuine setting:
\begin{theorem}\label{thm:first}
    \Cref{tab:nongenuine} gives a reflection presentation to the infinite families of non-genuine Steinberg crystallographic complex reflection groups.
\end{theorem}
Like in the genuine case, each presentation consists of a Coxeter-like diagram complemented by a \emph{single} extra order relation, now always in involutive form $w^2=1$ for a suitable reflection $w\in W^\alpha$.  
To the best of our knowledge, this pair of presentations is a first in the literature, thus acquiring independent interest: in particular, the work of Malle and Puente restricts only to genuine groups.
\begin{remark}\label{rmk:Frobenius}
    Unlike a non-crystallographic group which, as a \emph{whole}, is a complexification of an irreducible affine Weyl group, non-genuine crystallographic groups define non-trivial extensions of those Shepard-Todd groups that complexify a finite irreducible Weyl group. Faithful to history, it is the non-genuine class, back then under the name of complex crystallographic Coxeter (CCC) \cite{BS1978}, that set of the theory of infinite complex reflection groups in motion. An early remarkable application of CCC groups to the theory of Frobenius structures can be found in \cite{Dubrovin1996}.
\end{remark}

Let us first focus on the structure of the new diagrams.  
Comparing Table \ref{tab:nongenuine} against \ref{tab:genuine}, having moduli now requires one further additional node: the linear part's Coxeter diagram, of finite type $X\in\{A,C\}$, is affine-extended \emph{twice} to type $X^{(1)}$ like an inkblot. For $Q$ the root lattice, this visually encodes the group's signature structure
\begin{equation*}
    W^\alpha=\mathrm{Lin}(W)\ltimes\mathrm{Tran}^\alpha(W)=W(X)\ltimes Q(X\check{ \ })^{\times2},
\end{equation*}
whose translation part consists of two isomorphic lattices $Q(X\check{ \ })=\Lambda_1\simeq\Lambda_2=\alpha\Lambda_1$. Notice that the two additional nodes are precisely those not shared by the affine extensions, and connect via the new $x$-shaped lace. The shape represents an ``e$x$change rule'': even if the two additional generators satisfy no Coxeter relations, a special pair $(w_1,w_2)$ of equivalent words of same length exists such that each word contains each additional generator exactly once---but in exchanged order. We named $w_1=w_2$ the lace's $x$-relation. As its shape further suggests, the new lace must be considered as non-simple. 
Remarkably, these diagrams unlock straightforward descriptions of the commutator factor group (\Cref{tab:abel}), number of generating reflections (\Cref{prop:min-gen}), and conjugacy classes of hyperplanes (\Cref{prop:conj-class}), with all three features read off exactly à la Coxeter. In particular, both presentations are \emph{minimal} in the sense of number of generators.  

As this diagrammatics suggests, the presentation's signature facet is the $x$-relation. We point out two striking features of it. On the one hand, in either side of the relation, only the generators (on top of the two additional ones) spanned by any $3$-cycle\footnote{A cycle's segment must `` tangle'' the reflections: commutation (no lace) and independence ($\infty$-lace) are disallowed.} involving the exchange lace appear. On the other hand, despite all generators being involutions, inverse letters appear explicitly. 
The latter is a topological manifestation: for the presentation to have any hope in validating the braid theorem, the $x$-relation must be written as if reflections were braid generators. However, being the $x$-relation not Coxeter, no a priori way to do so is known. Nevertheless, a deductive strategy is offered by an explicit description of the group's regular orbit space:
\begin{lemma}
    Let $W^\alpha_n$ be an infinite family of non-genuine Steinberg crystallographic complex reflection groups. Then, the regular orbit space of $W^\alpha$ is a configuration space for all $n$.
\end{lemma}
 In particular, type $A$ corresponds to special configurations of $n$ points on the torus \cite{IR2025}, while type $C$ corresponds to configurations of $n$ points on the four-punctured sphere (see \Cref{sec:braid}). 
With a presentation available for both configuration spaces, we succeed in unraveling the exact ``topological form'' of each $x$-relation as reported in \Cref{tab:nongenuine}. Notice that no square of a generator is involved.
These topology-driven expressions, for which we also provide a transparent geometric understanding in the language of configuration spaces (Figures \ref{fig:ell-braid} and \ref{fig:braid2}), are foundational to our second main result, which completes the extension of \Cref{thm:genuine-pres} to the non-genuine setting:
\begin{theorem}\label{thm:second}
    Both reflection presentations in \Cref{tab:nongenuine} achieve the braid theorem.
\end{theorem}
\begin{remark}\label{rmk:conj}
As it happens, the braid presentation in type $A$ matches the central evaluation of van der Lek's \emph{extended Artin group} $\tilde{\mathrm{A}}\mathrm{r}(A^{(1)}_n)$. Following his thesis \cite{vanderLek1983}, this is an \emph{abstract} group obtained by deleting order relations and same-index ``pushrelations'' from a special presentation of Saito-Takebayashi's \cite{ST1997} hyperbolic (central) extension $\tilde{W}(A^{(1,1)}_n)\simeq W(A_n^{(1)})\ltimes Q(A_n^{(1)})$ of the elliptic Weyl group $W(A^{(1,1)}_n)\simeq W(A_n)\ltimes Q(A_n)^{\times2}$. Crucially, the extended Artin group matches the fundamental group of the regular orbit space for $\tilde{W}(A^{(1,1)}_n)$. Actually, this holds for any finite type $X$ and should be thought of as a ``coarser'' braid theorem---being the removal not limited to the order relations. As the generator of the center remains central in the extended Artin group, one is invited to ask whether the central evaluation of $\tilde{\mathrm{A}}\mathrm{r}(X^{(1)})$ accordingly matches the braid group of $W(X^{(1,1)})$. 
This statement appears to be specialists' folklore: failing to find any actual proof in the literature, we leave it as Conjecture \ref{conj:vdL}. While waiting for a systemic argument, \Cref{lem:braidonezero} gives an explicit proof in type $A$ and \Cref{rmk:IonSahi2} confirms type $C$.
\end{remark}

At last, the lemma's topological insight suffices to complete our initial program: for any rank $n$, the missing GDAHA of type $D^{(1)}_4$ deforms precisely the configuration space in type $C$ corresponding to $[G(2,1,n)]_1^\alpha$, and the isomorphism we constructed in types $E^{(1)}_{6,7,8}$ admits the analogous generalization.
En passant, we also confirm that the deformation in type $A$ extends the identification we proved in topology: namely, the generic Hecke algebra of $[W(A_{n-1})]^\alpha$, $n\geq3$, matches the central evaluation of the EHA in type $A_{n-1}^{(1,1)}$ .

All considered, we obtain our third and final main result:
\begin{theorem}\label{thm:third}
    Let $G[d,p,n]$ denote the Shepard-Todd infinite family. For $p=1$ and natural lattice of label $k=1$, the generic Hecke algebras deforming the Steinberg infinite families of crystallographic complex reflection groups are isomorphic to the GDAHAs of affine type.  
\end{theorem}
Since Popov's classification involves either infinite families or sporadic groups, and Steinberg families $[G(d,p,n)]_k$ come as subgroups of $[G(d,1,n)]_1$ \cite{Malle1996}, we should rethink GDAHA as nothing but the most ``general'' chapter in the Hecke-deformation theory of crystallographic complex reflection groups. The rest of the theory's book is waiting to be written.
\begin{remark}
    An obvious first step forward would be to define, by generators and relations, the Hecke algebras attached to the remaining families and sporadic groups, with a preliminary hunt for a reflection presentation achieving the braid theorem. Alternatively, one could attempt a more conceptual approach based on KZ connections akin to that developed by Broué, Malle, and Rouquier \cite{BMR1998}. On this note, one naturally asks whether the ``BMR freeness conjecture'', namely the algebraic flatness of the deformation, extends to the complex crystallographic world. Considering that the original statement, now finally confirmed for all Shepard-Todd groups, kept the community engaged for two decades, we expect such relaunch of the conjecture to be of great interest. E.g., this conjecture for affine GDAHAs was stated right when these algebras were first introduced in 2006 \cite{EGO2006}, and is open since. 
\end{remark}

The present paper is organized as follows:

\Cref{sec:pres} contains the geometric derivation of the new reflection presentations, which are then used to prove the anticipated triple of properties for the non-genuine infinite families;

\Cref{sec:braid} wades through topology to prove the braid theorem for both presentations;

\Cref{sec:Hecke} concludes the paper, introducing the Hecke algebra before proving the sought-after reboot of GDAHA.

Finally, Appendices \ref{app1} and \ref{app2} visualize, respectively in the non-genuine and genuine cases, the reflection presentations achieving the braid theorem.

\paragraph{Acknowledgments}
\!\!\!The author is truly grateful to both Yoshihisa Saito and Bogdan Ion for the fruitful discussions. This research was funded by the Japan Society for the Promotion of Science [PE24720], and Fondazione Cariparo [C93C22008360007] via University of Padua.

\paragraph{Conventions}
\!\!\!The following alphabetic notations are used throughout the paper:
\begin{itemize}
    \item Greek letters $\sigma, \tau, \upsilon$ for reflection elements;
    \item Latin lowercase letters $s, t, u$ for braid elements;
    \item Latin uppercase letters $S, T, U$ for Hecke elements;
    \item blackboard letters $\mymathbb{s}, \mymathbb{t}, \mymathbb{u}$ for Hecke parameters.
\end{itemize}

Moreover, when depicting $n$ braids on a thickened surface $\Sigma_{g,s}\times I$ with strata in the interval $I$ having genus $g$ and $s$ holes, we use the following pictures:
\begin{itemize}
    \item top view (left) with same-colored edges identified, and side view (right) for $\Sigma_{0,s}$ as
\begin{figure}[!ht]
    \centering
    \includegraphics[height=8em]{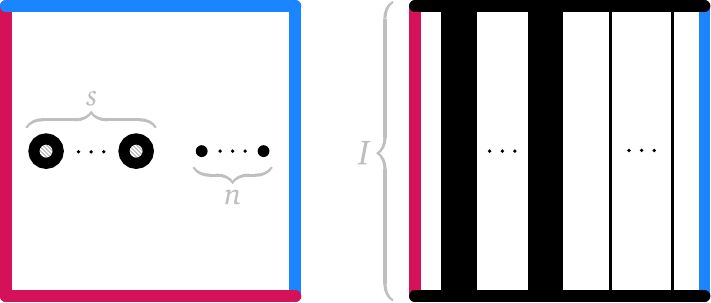}
\end{figure}
\item top view (left) with opposite edges identified to form the {\color{2blue}longitude} and {\color{2red}meridian} cycles, meridional side view (center), and longitudinal side view (right) for $\Sigma_{1,0}=\mathbb{T}$ as
\begin{figure}[!h]
    \centering
    \includegraphics[height=8em]{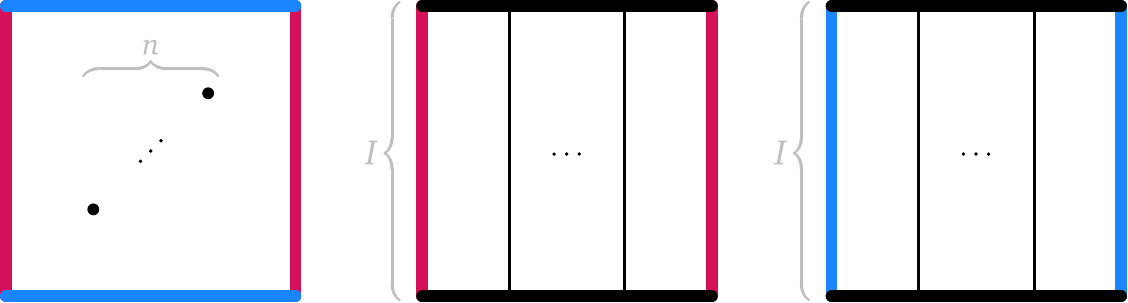}
\end{figure}
\end{itemize}

\section{Coxeter-like minimal presentations}\label{sec:pres}

We start the section by deriving reflection presentations for the pair of non-genuine families \eqref{non-genuine}. Let us preliminarily point out that, as an irreducible group, $G(1,1,n)\simeq W(A_{n-1})$ by restriction to the standard subspace $\langle e_1-e_2,\ldots,e_{n-1}-e_{n}\rangle\simeq\complex^{n-1}$ \cite[§4]{PS2019}. In light of this clarification, we hereafter denote by $[W(A_{n-1})]^\alpha$ the corresponding family. Note also that $G(2,1,n)\simeq W(C_n)$ and $A_1 = C_1$ manifests in the classification as $[W(A_1)]^\alpha\simeq[G(2,1,1)]_1^\alpha$.

Explicitly, the presentations claimed in  \Cref{thm:first} read as in
\begin{theorem}The following isomorphisms hold:
\begin{equation}\label{A1-pres}
     [G(2,1,1)]^\alpha_1 \ \simeq \ [W(A_1)]^\alpha \ \simeq \ \big\langle\ \sg{1},\sg{2},\sg{3}\ \big|\ \sg{1}^2=\sg{2}^2=\sg{3}^2=(\sg{1}\sg{2}\sg{3})^2=1 \ \big\rangle,
\end{equation}
\begin{equation}\label{C-pres}
    [G(2,1,n)]^\alpha_1 \ \simeq \ \left\langle\ \sg{1},\cdots,\sg{n+2}\ \left|\ \begin{array}{c}
        \sg{1}^2=\sg{2}^2=\cdots=\sg{n+2}^2=(\sg{1}\cdots\sg{n+2}\sg{n}\cdots\sg{2})^2=1,\\
        \sg{i}\sg{j}=\sg{j}\sg{i}  \text{\hspace{1.5em}for\hspace{.5em}}  i-j\neq1,\ 1\leq i,j\leq n,\hfill\\\phantom{\sg{i}\sg{j}=\sg{j}\sg{i}\hspace{1.75em}}\text{or\hspace{.5em}}i=n+1,n+2,\ 1\leq j\leq n-1,\hfill\\
        \sg{i}\sg{j}\sg{i}=\sg{j}\sg{i}\sg{j}  \text{\hspace{1.5em}for\hspace{.5em}}  i-j=1 ,\ 2\leq i,j\leq n,\hfill\\
        \sg{i}\sg{j}\sg{i}\sg{j}=\sg{j}\sg{i}\sg{j}\sg{i} \text{\hspace{1.5em}for\hspace{.5em}}  i=1,\ j=2,\hfill\\
        \phantom{\sg{i}\sg{j}\sg{i}\sg{j}=\sg{j}\sg{i}\sg{j}\sg{i}\hspace{1.75em}}\text{or\hspace{.5em}}i=n+1,n+2,\ j=n,\hfill\\        \sg{n}\sg{n+1}\sg{n}^{-1}\sg{n+2}=\sg{n+2}\sg{n}\sg{n+1}\sg{n}^{-1}\hfill    
    \end{array}\right.\right\rangle,
\end{equation}
\begin{equation}\label{A-pres}
    [W(A_{n-1})]^\alpha \ \simeq \ \left\langle\ \sg{1},\cdots,\sg{n+1}\ \left|\ \begin{array}{c}
    \sg{1}^2=\cdots=\sg{n+1}^2=(\sg{1}\cdots\sg{n+1}\sg{n-1}\cdots\sg{2})^2=1,\hfill\\
    \sg{i}\sg{j}=\sg{j}\sg{i}  \text{\hspace{1.5em}for\hspace{.5em}}  i-j\neq1,\ 1\leq i,j\leq n-1,\hfill\\        \phantom{\sg{i}\sg{j}=\sg{j}\sg{i}\hspace{1.75em}}\text{or\hspace{.5em}}i=n,n+1,\ 2\leq j \leq n-2,\hfill\\
        \sg{i}\sg{j}\sg{i}=\sg{j}\sg{i}\sg{j}  \text{\hspace{1.5em}for\hspace{.5em}}  i-j=1 ,\ 1\leq i,j\leq n-1,\hfill\\    \phantom{\sg{i}\sg{j}\sg{i}=\sg{j}\sg{i}\sg{j}\hspace{1.75em}}\text{or\hspace{.5em}}i=n,n+1,\ j=1,n-1,\hfill\\
        \sg{n+1}\sg{n-1}\cdots\sg{1}\sg{n}\isg{1} =\sg{n-1}\isg{n}\isg{n-1}\cdots\isg{1}\isg{n+1}   
    \end{array}\right.\right\rangle.\vspace{.75em}
\end{equation}
\end{theorem}
\begin{proof}
For $W_n^\alpha$ a non-genuine family of crystallographic complex reflection groups, a first presentation appears by adjoining to that of $\Lin$ a basis of the lattice $\Tran$ as additional generators, together with both the relations capturing the action $\Lin\triangleright\Tran$ and commutations in $\Tran$. 
This intermediate presentation is then finalized by replacing the additional generators with suitable reflections, while carefully keeping track of the relations.
We denote an element in $W_n^\alpha$ as $(g\,|\,t)$, for $g\in\Lin$ and $t\in\Tran$.

\paragraph{Case $W^\alpha=[G(2,1,1)]^\alpha_1\simeq[W(A_1)]^\alpha$\\}
We start with the simplest $[G(2,1,1)]^\alpha_1\simeq\integer_2\ltimes\integer^2$, whose linear part's Coxeter diagram
\begin{figure}[!h]
    \centering
    \includegraphics[height=1.25em]{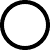}
\end{figure}\vspace{-.25em}
\\of type $A_1$ corresponds to the \emph{distinguished reflection} \cite{BMR1998} $\sg{1}:=e^{2\pi i/2}=-1$, and translation part's lattice can be written as $\Lambda^\alpha:=\langle\tu{1},\tu{2}\rangle=\langle1,\alpha\rangle$.
Defining the new pair of reflections
\begin{equation}\label{G2sigmas}
    \sg{2}=\tu{1}\sg{1}=(-1\,|\,1), \qquad \sg{3}=\tu{2}\sg{1}=(-1\,|\,\alpha), 
\end{equation}
it is immediate to see that no relations occur between any two of $\{\sg{1},\sg{2},\sg{3}\}$, in that their products belong to the lattice.
E.g., $\sg{1}\sg{2}=(1,-1)=\tu{1}^{-1}$ and, $\forall\ i\in\integer$, $(\sg{1}\sg{2})^i$ can match neither the identity nor any reflections. Nevertheless, $\sg{1}\sg{2}\sg{3}=(-1\,|\,\alpha-1)$ is itself a reflection and we get precisely the Coxeter diagram\vspace{-1em}
\begin{figure}[!h]
    \centering
    \includegraphics[height=6em]{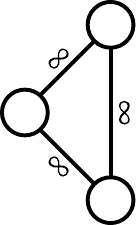}
\end{figure}\vspace{-.25em}
\\in \Cref{tab:nongenuine}---up to the order relation 
\begin{equation}\label{extra}
    (\sg{1}\sg{2}\sg{3})^2=1.
\end{equation}
It is then straightforward to check that all relations among $\sg{1},\sg{2},\tu{1},\tu{2}$ follow from the diagram together with \eqref{extra}, with the latter encapsulating the commutation $\tu{1}\tu{2}=\tu{2}\tu{1}$.

Note that, adding $\sg{4}:=(\sg{1}\sg{2}\sg{3})^{-1}$, the reflection presentation reads equivalently as
\begin{equation}
    \left\langle\ \sg{1},\sg{2},\sg{3},\sg{4}\ \left|\ \begin{array}{c}
        \sg{1}^2=\sg{2}^2=\sg{3}^2=\sg{4}^2=1,\\
        \sg{1}\sg{2}\sg{3}\sg{4}=1\hfill
    \end{array}\right.\right\rangle
\end{equation}
for a tetrahedronal analogue of the triangular diagram and the extra relation swapped to braid type---by a closedness condition, cf. the first relation in presentation \eqref{p1sigma04} at $n=1$.

As expected, $[W(A_1)]^\alpha$ yields the same result. Indeed, for the permutation
\begin{equation}
    \sigma_1=\begin{pmatrix}
        0 & 1 \\ 1 & 0
    \end{pmatrix}
\end{equation}
and the standard $A$-type (double) lattice
\begin{equation}
    \Lambda^\alpha:=\Big\langle\begin{pmatrix}1\\-1\end{pmatrix},\begin{pmatrix}\alpha\\-\alpha\end{pmatrix}\Big\rangle,
\end{equation}
take reflections
\begin{equation}\label{G1sigmas}
    \sg{2}=\tu{1}\sg{1}=\left(\begin{array}{@{}cc|c@{}} 0 & 1 & 1\\1 & 0 & -1\end{array}\right), \qquad \sg{3}=\tu{2}\sg{1}=\left(\begin{array}{@{}cc|c@{}} 0 & 1 & \alpha\\1 & 0 & -\alpha\end{array}\right)
\end{equation}
exactly as before to check analogously that no relations occur beyond the order one
\begin{equation}
    (\sg{1}\sg{2}\sg{3})^2=1.
\end{equation}

\paragraph{Case $W_n^\alpha=[G(2,1,n)]^\alpha_1,\ n\geq2$\\}
We scale formulae \eqref{G2sigmas} in rank. On the one hand, $G(2,1,n)\simeq W(C_n)$ has Coxeter diagram
\begin{figure}[!h]
    \centering
    \includegraphics[height=1.25em]{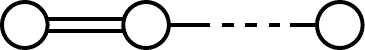}
\end{figure}\vspace{-.25em}
\\with corresponding generators $\sg{1}=\diag(-1,1,\ldots,1)$ and permutation matrices $\sg{i}=(i-1,i)$, $2 \leq i \leq n$. On the other hand, denoting by $\{e_1,\ldots,e_n\}$ the standard basis of $\complex^n$, $\Lambda^\alpha:=\langle\tu{1},\ldots\tu{2n}\rangle$ for the assignments $\tu{2i-1}=e_i$, $\tu{2i}=\alpha e_i$.
The new pair of reflections, reading now as
\begin{align*}
    &\sg{n+1}=\sg{n}\cdots\sg{2}\tu{1}\sg{1}\cdots\sg{n}=
    \left(\begin{array}{@{}cc|c@{}}
         \mymathbb{1}_{n-1} & 0 & 0\\
         0 & -1 & 1
    \end{array}\right),\allowdisplaybreaks\\
    &\sg{n+2}=\sg{n}\cdots\sg{2}\tu{2}\sg{1}\cdots\sg{n}=
    \left(\begin{array}{@{}cc|c@{}}
         \mymathbb{1}_{n-1} & 0 & 0\\
         0 & -1 & \alpha
    \end{array}\right),
\end{align*}
manifestly commute with all but the last permutation matrix, and it is immediate to check that $(\sg{n}\sg{n+1})^4=1=(\sg{n}\sg{n+2})^4$.
Again, $\sg{n+1}\sg{n+2}\in\Lambda^\alpha$ and the previous argument implies no Coxeter relations between them exist, serving the Coxeter diagram\vspace{-.25em}
\begin{equation}\label{inf2Cn}
\begin{matrix}\includegraphics[height=6em]{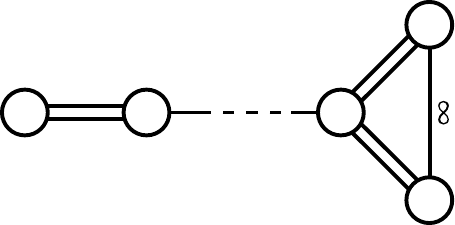}\end{matrix}\vspace{-.5em}
\end{equation}
\!As anticipated, this diagrammatic description is incomplete and fails to encode the whole reflection geometry: explicit matrix computations reveal a missing braid-type quartic relation
\begin{equation}\label{braid}
    \sg{n}\sg{n+1}\sg{n}\sg{n+2}=\sg{n+2}\sg{n}\sg{n+1}\sg{n}
\end{equation}
which, following the Introduction, turns the graph into the Coxeter-like diagram in \Cref{tab:nongenuine}. Which generating reflection conceals its \emph{inverse} braid generator, thus ruling the Table's exact form of the corresponding $x$-relation, is determined a posteriori in the proof of the braid theorem.
Notice that, as anticipated, the reflections appearing on each side of \eqref{braid} are precisely those spanned by the unique $3$-cycle involving the two incident double laces.   
Moreover, one quickly checks that
\begin{equation}
    \sg{1}\cdots\sg{n+2}\sg{n}\cdots\sg{2}=
    \left(\begin{array}{@{}cc|c@{}}
         -1 & 0 & \alpha-1\\
         0 & \mymathbb{1}_{n-1} & 0
    \end{array}\right),
\end{equation}
yields the order relation
\begin{equation}\label{cyclic}
    (\sg{1}\cdots\sg{n+2}\sg{n}\cdots\sg{2})^2=1.
\end{equation}

It is then easy to verify that all relations between $\sg{1},\cdots\sg{n},\tu{1},\cdots,\tu{2n}$, where
\begin{equation*}
    \tu{2i-1}=\sg{i+1}\cdots\sg{n+1}\cdots\sg{1}\cdots\sg{i}, \qquad \tu{2i}=\sg{i+1}\cdots\sg{n}\sg{n+2}\sg{n}\cdots\sg{1}\cdots\sg{i},
\end{equation*}
follow from the Coxeter-like diagram paired with its complementary order relation.
Indeed, the actions $\sigma_i\triangleright\tu{j}$ are straightforward to check. We also give the basic manipulations needed to prove the commutations, for the $n=3$ case involving all core types of relations, as they easily adapt with increasing rank and lattice size:
\begin{align*}
\tu{1}\tu{2}&=\sg{2}\sg{3}(\sg{4}\sg{3}\sg{2}\sg{1}\cdot\sg{2}\sg{3}\sg{5})\sg{3}\sg{2}\sg{1}\overset{\eqref{cyclic}}{=}\sg{2}\sg{3}\sg{5}\sg{3}\sg{2}\sg{1}\sg{2}\sg{3}\sg{4}\sg{3}\sg{2}\sg{1}=\tu{2}\tu{1},\allowdisplaybreaks\\
\tu{1}\tu{3}&=\tu{1}\cdot\sg{2}\triangleright\tu{1}=(\tu{1}\sg{2})\tu{1}\sg{2}=\sg{2}\tu{1}(\tu{2}\sg{2})=\sg{2}\tu{1}\sg{2}\tu{1}=\tu{3}\tu{1},\allowdisplaybreaks\\
\tu{1}\tu{4}&=\tu{1}\cdot\sg{2}\triangleright\tu{2}=\tu{4}\tu{1},\allowdisplaybreaks\\
\tu{1}\tu{5}&=\tu{1}\cdot\sg{3}\triangleright\tu{3}=\tu{1}\sg{3}\tu{3}\sg{3}=\sg{3}\tu{3}\sg{3}\tu{1}=\tu{4}\tu{1},\allowdisplaybreaks\\
\tu{1}\tu{6}&=\tu{1}\cdot\sg{3}\triangleright\tu{4}=\tu{6}\tu{1},\allowdisplaybreaks\\
\tu{2}\tu{3}&=\sg{2}\sg{3}\sg{5}(\sg{3}\sg{2}\sg{3})\sg{4}\sg{3}(\sg{1}\sg{2}\sg{1}\sg{2})=(\sg{2}\sg{3}\sg{2})\sg{5}\sg{3}\sg{4}\sg{2}\sg{3}\sg{2}\sg{1}\sg{2}\sg{1}=\\&=\sg{3}\sg{2}(\sg{3}\sg{5}\sg{3}\sg{4})\sg{2}\sg{3}\sg{2}\sg{1}\sg{2}\sg{1}\overset{\eqref{braid}}{=}\sg{3}\sg{4}\sg{2}\sg{3}\sg{5}(\sg{3}\sg{2}\sg{3})\sg{2}\sg{1}\sg{2}\sg{1}=\\&=\sg{3}\sg{4}(\sg{2}\sg{3}\sg{2})\sg{5}\sg{3}\sg{1}\sg{2}\sg{1}=\sg{3}\sg{4}\sg{3}\sg{2}\sg{1}\sg{3}\sg{5}\sg{3}\sg{2}\sg{1}=\tu{3}\tu{2},\allowdisplaybreaks\\
\tu{2}\tu{4}&=\sg{2}\sg{3}\sg{5}(\sg{3}\sg{2}\sg{3})\sg{5}\sg{3}(\sg{1}\sg{2}\sg{1}\sg{2})=(\sg{2}\sg{3}\sg{2})\sg{5}\sg{3}\sg{5}\sg{2}\sg{3}\sg{2}\sg{1}\sg{2}\sg{1}=\\&=\sg{3}\sg{2}(\sg{3}\sg{5}\sg{3}\sg{5})\sg{2}\sg{3}\sg{2}\sg{1}\sg{2}\sg{1}=\sg{3}\sg{5}\sg{2}\sg{3}\sg{5}(\sg{3}\sg{2}\sg{3})\sg{2}\sg{1}\sg{2}\sg{1}=\\&=\sg{3}\sg{5}(\sg{2}\sg{3}\sg{2})\sg{5}\sg{3}\sg{1}\sg{2}\sg{1}=\sg{3}\sg{5}\sg{3}\sg{2}\sg{1}\sg{3}\sg{5}\sg{3}\sg{2}\sg{1}=\tu{4}\tu{2},\allowdisplaybreaks\\
\tu{2}\tu{5}&=\tu{2}\cdot\sg{3}\triangleright\tu{3}=\tu{2}\sg{3}\tu{3}\sg{3}=\sg{3}\tu{3}\sg{3}\tu{2}=\tu{5}\tu{2},\allowdisplaybreaks\\
\tu{2}\tu{6}&=\tu{2}\cdot\sg{3}\triangleright\tu{4}=\tu{6}\tu{2},
\end{align*}
and those remaining proved similarly.
\begin{remark}\label{rmk:IonSahi1}
    For $n=2$, an alternative minimal presentation was defined in \cite[Theorem 2]{KTY1982} but is far from Coxeter-like, as it requires a \emph{triple} of additional braid-type relations. The Coxeter diagram \eqref{inf2Cn} has already appeared \cite[Figure 2]{IS2020} in the study of extended Artin groups $\tilde{\mathrm{A}}\mathrm{r}(X^{(1)})$, loc. cit. under the name of double affine Artin groups. In particular \cite[Proposition 5.6]{IS2020}, the abstract group $\tilde{\mathrm{A}}\mathrm{r}(C^{(1)})$ is presented, up to a further central generator $\dutchcal{C}$, as a quotient of the Artin group attached to \eqref{inf2Cn} by the very $x$-relation $\sg{n}\sg{n+1}\sg{n}^{-1}\sg{n+2}=\sg{n+2}\sg{n}\sg{n+1}\sg{n}^{-1}$ the authors referred to as ``elliptic braid relation''.
\end{remark}

\paragraph{Case $W_n^\alpha=[W(A_{n-1})]^\alpha,\ n\geq3$\\}
We now scale \eqref{G1sigmas} in rank. On the one hand, $G(1,1,n)\simeq\mathfrak{S}_n$ has irreducible reflection presentation that matches that of $W(A_{n-1})$ via the permutation matrices $\sg{i}=(i,i+1)$, $1 \leq i \leq n-1$, visualized by the Coxeter diagram\vspace{-1em}
\begin{figure}[!h]
    \centering
    \includegraphics[height=1.25em]{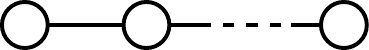}
\end{figure}\vspace{-.25em}
\\of type $A$.
On the other hand, we replace the standard lattice of type $A$, generated by the differences $e_i-e_{i+1}$, with the equivalent $\Lambda^\alpha:=\langle\tu{1},\ldots\tu{2(n-1)}\rangle$ for $\tu{2i-1}=e_1-e_{i+1}$ and $\tu{2i}=\alpha\tu{2i-1}$.
Then, the new pair of reflections is defined as
\begin{align*}
    &\sg{n}=\sg{n-1}\cdots\sg{2}\tu{1}\sg{1}\cdots\sg{n-1}=
    \left(\begin{array}{@{}ccc|c@{}}
         0 & 0 & 1 & 1\\
         0 & \mymathbb{1}_{n-2} & 0 & 0\\
         1 & 0 & 0 & -1
    \end{array}\right),\allowdisplaybreaks\\
    &\sg{n+1}=\sg{n-1}\cdots\sg{2}\tu{2}\sg{1}\cdots\sg{n-1}=
    \left(\begin{array}{@{}ccc|c@{}}
         0 & 0 & 1 & \alpha\\
         0 & \mymathbb{1}_{n-2} & 0 & 0\\
         1 & 0 & 0 & -\alpha
    \end{array}\right),
\end{align*}
and each element is easily checked to turn the diagram affine.
Once again, $\sg{n}\sg{n+1}\in\Lambda^\alpha$ serves the Coxeter diagram\vspace{-1em}
\begin{figure}[!h]
    \centering
    \includegraphics[height=6em]{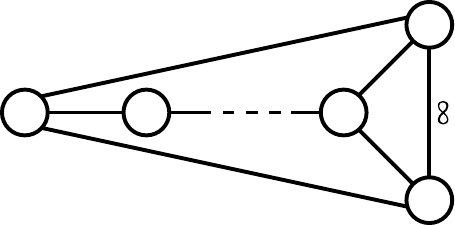}
\end{figure}\vspace{-.25em}
\\whose missing braid-type relation involves now $n+2$ generators per side:
\begin{equation}\label{braid2}
    \sg{n+1}\sg{n-1}\cdots\sg{1}\sg{n}\sg{1}=\sg{n-1}\sg{n}\sg{n-1}\cdots\sg{1}\sg{n+1}.
\end{equation}
For the explicit topological $x$-relation form of \eqref{braid2} delivering the diagram in \Cref{tab:nongenuine}, the analogous remarks made above in type $C$ hold. 
Furthermore, one can quickly check that
\begin{equation}
    \sg{1}\cdots\sg{n+1}\sg{n-1}\cdots\sg{2}=\left(\begin{array}{@{}c|c@{}}
         & \alpha-1\\
       \sg{1}  & 1-\alpha\\
           & 0
    \end{array}\right)
\end{equation}
yields the order relation
\begin{equation}\label{cyclic2}
    (\sg{1}\cdots\sg{n+1}\sg{n-1}\cdots\sg{2})^2=1.
\end{equation}

It is then easy to verify that all relations between $\sg{1},\cdots\sg{n-1},\tu{1},\cdots,\tu{2n}$, where
\begin{equation*}
    \tu{2i-1}=\sg{i+1}\cdots\sg{n}\cdots\sg{1}\cdots\sg{i}, \qquad \tu{2i}=\sg{i+1}\cdots\sg{n-1}\sg{n+1}\sg{n-1}\cdots\sg{1}\cdots\sg{i},
\end{equation*}
follow from the twice-extended diagram together with the two extra relations. Analogously to the previous case, the actions are straightforward to check and it suffices to give the $n=4$ basic manipulations proving the commutations:
\begin{align*}
\tu{1}\tu{2}&=\sg{2}(\sg{3}\sg{4}\sg{3}\sg{2}\sg{1}\cdot\sg{2}\sg{3}\sg{5})\sg{3}\sg{2}\sg{1}\overset{\eqref{cyclic2}}{=}\sg{2}\sg{3}\sg{5}\sg{3}\sg{2}\sg{1}\cdot\sg{2}\sg{3}\sg{4}\sg{3}\sg{2}\sg{1}=\tu{2}\tu{1},\allowdisplaybreaks\\
\tu{1}\tu{3}&=\sg{2}(\sg{3}\sg{4}\sg{3})\sg{2}\sg{1}\cdot\sg{3}\sg{4}\sg{3}\sg{2}\sg{1}\sg{2}=\sg{4}(\sg{2}\sg{3}\sg{2})\sg{4}(\sg{1}\sg{3})\sg{4}\sg{3}\sg{2}\sg{1}\sg{2}=\\&=\sg{4}\sg{3}\sg{2}(\sg{3}\sg{4}\sg{3})\sg{1}\sg{4}\sg{3}\sg{2}\sg{1}\sg{2}=(\sg{4}\sg{3}\sg{4})\sg{2}\sg{3}(\sg{4}\sg{1}\sg{4})\sg{3}\sg{2}\sg{1}\sg{2}=\\&=\sg{3}\sg{4}\sg{3}\sg{2}(\sg{3}\sg{1})\sg{4}(\sg{1}\sg{3})\sg{2}\sg{1}\sg{2}=\sg{3}\sg{4}\sg{3}\sg{2}\sg{1}\sg{3}\sg{4}\sg{3}(\sg{1}\sg{2}\sg{1})\sg{2}=\tu{3}\tu{1},\allowdisplaybreaks\\
\tu{1}\tu{4}&=\sg{2}\sg{3}\sg{4}\sg{3}\sg{2}(\sg{1}\cdot\sg{3})\sg{5}\sg{3}\sg{2}\sg{1}\sg{2}=\sg{2}\sg{3}\sg{4}(\sg{3}\sg{2}\sg{3})\sg{1}\sg{5}\sg{3}\sg{2}\sg{1}\sg{2}=\\&=(\sg{2}\sg{3}\sg{2})\sg{4}\sg{3}\sg{2}\sg{1}\sg{5}\sg{3}\sg{2}\sg{1}\sg{2}=\sg{3}\sg{2}(\sg{3}\sg{4}\sg{3}\sg{2}\sg{1}\sg{5})\sg{3}\sg{2}\sg{1}\sg{2}\overset{\eqref{braid2}}{=}\\&=\sg{3}\sg{5}(\sg{2}\sg{3}\sg{2})\sg{1}\sg{4}\sg{3}(\sg{1}\sg{2}\sg{1})\sg{2}=\sg{3}\sg{5}\sg{3}\sg{2}\sg{1}\sg{3}\sg{4}\sg{3}\sg{2}\sg{1}=\tu{4}\tu{1},\\
\tu{1}\tu{5}&=\tu{1}\cdot\sg{3}\triangleright\tu{3}=\tu{1}\sg{3}\tu{3}\sg{3}=\sg{3}\tu{3}\sg{3}\tu{1}=\tu{5}\tu{1},\allowdisplaybreaks\\
\tu{1}\tu{6}&=\tu{1}\cdot\sg{3}\triangleright\tu{4}=\tu{1}\sg{3}\tu{4}\sg{3}=\sg{3}\tu{4}\sg{3}\tu{1}=\tu{6}\tu{1},
\end{align*}
with those remaining proved similarly.
\end{proof}

\subsection{Abelianization, number of generators, and orbits of hyperplanes}

In this section, we take advantage of our newly found diagrams to determine the minimal number of generators and classify reflections in the non-genuine setting.
Remarkably, the diagrams turn out to fully encapsulate both properties precisely in the Coxeter yoga: on the one hand, one simply counts the nodes; on the other hand, there is a (order $2$) conjugacy class for any connected component after removal of even and $x$-laces---plus one from the extra order relation in type $C$.
Furthermore, they immediately allow to describe the abelianization: detailing the order of each cyclic factor, one gets the following
\begin{table}[!h]
\begin{tabular}{ccc}\toprule
$\hspace{-.5em}W$ & \hspace{3em} & $W/[W,W]$\\\midrule\\[-.5em]
$[W(A_1)]^\alpha$ & & $2\times2\times2$\\[.5em]
$[W(A_{n-1})]^\alpha$, $n\geq3$ & & $2$\\[.5em]\midrule\\[-.5em]
$[G(2,1,n)]_1^\alpha$, $n\geq2$ & & $2\times2\times2\times2$\\[.5em]\bottomrule
\end{tabular}
\caption{Structure of the commutator factor group for the infinite families of non-genuine Steinberg complex crystallographic reflection groups.}
\label{tab:abel}
\end{table}

Let us start by denoting with $m(W)$ the minimal number of reflections needed to generate a given reflection group $W$.
\begin{proposition}\label{prop:min-gen}
For the non-genuine Steinberg infinite families,
\begin{equation}
    m(W^\alpha_n)=m(\mathrm{Lin}(W^\alpha_n))+2.    
\end{equation}
\end{proposition}
\begin{proof}    
On the one hand, each diagram's node counting gives an upper bound. For $[W(A_1)]^\alpha$, the order $2^3$ of the commutator factor group suffices to prove the assertion by requiring a generating set to contain at least $3$ elements. 
On the other hand, assume $[G(2,1,n)]_1^\alpha$ is generated by $n$ reflections $\sg{1},\ldots,\sg{n}$ whose linear parts match the standard 
representation of $G(2,1,n)$. Being the latter irreducible, $\langle\sg{1},\ldots,\sg{n}\rangle$ fixes the origin and is thus finite.  
In particular, $\langle\sg{1},\ldots,\sg{n}\rangle\simeq G(2,1,n)$ as any element in the kernel, being a pure translation, would contradict finiteness. Analogously, one obtains that $n-1$ reflections in $[W(A_{n-1})]^\alpha$ generate a group isomorphic to $\mathfrak{S}_n$.

Since $[G(2,1,n)]_1^\alpha$ has abelianization of order $2^4$ and linear part with just two classes of reflections, we get a lower bound confirming the assertion. For $[W(A_{n-1})]^\alpha$, abelianization is of no use and we can just deduce the insufficient lower bound of $n$ generating reflections.
Nevertheless, as already observed in the Introduction, $[W(A_{n-1})]^\alpha\simeq W(A_{n-1})\ltimes Q(A_{n-1})^{\times2}$ manifestly contains $W(A^{(1)}_{n-1})\simeq W(A_{n-1})\ltimes Q(A_{n-1})$ as a \emph{proper} subgroup, whose minimal generating set requires exactly $n$ elements.   
\end{proof}

Extending \cite[§6]{Malle1996} to our non-genuine setting, we next classify reflections up to the relation that the corresponding hyperplanes of fixed points lie in the same (order $2$) orbit:
\begin{proposition}\label{prop:conj-class}
    There are four conjugacy classes in $[G(2,1,1)]_1^\alpha\simeq[W(A_1)]^\alpha$. For $n\geq2$, there are five conjugacy classes in $[G(2,1,n)]_1^\alpha$ and a single conjugacy class in $[W(A_{n})]^\alpha$.
\end{proposition}
\begin{proof}
We set embeddings $W_n^\alpha \hookrightarrow GL_{n+1}(\complex)$ so to send the linear parts of the generators to the known representatives for the classes of reflections in the finite group $\mathrm{Lin}(W_n^\alpha)$, and denote again permutation matrices 
by $(i,i+1)$.
For $W_n^\alpha=[G(2,1,n)]^\alpha_1,\ n\geq1$, the embedding reads as 
\begin{equation*}
\begin{matrix}
    \sg{1} & \mapsto & \mathrm{diag}(1,-1,1,\ldots,1),\hfill\\
    \sg{i} & \mapsto & (i,i+1),\quad 2\leq i \leq n.
\end{matrix}
\end{equation*}
Thus, an affine reflection $\sigma \in W_n^\alpha$ has $n\times n$ linear part that is conjugate to either $\mathrm{diag}(-1,1,\ldots,1)$ or, say, $(1,2)$. In the first case,
\begin{equation*}
    \sigma=\begin{pmatrix}
        1 & 0 & 0 \\
        b & -1 & 0 \\
        0 & 0 & \mymathbb{1}_{n-1}
    \end{pmatrix}
\end{equation*}
for $b\in\integer+\alpha\integer$. Conjugating by suitable translations, we can restrict to having $b\in\{0,1,\alpha,1+\alpha\}$: respectively, these values cover the combinations $2z_1+i2z_2$, $(1+2z_1)+i2z_2$, $(1+2z_1)+i(1+2z_2)$, and $2z_1+i(1+2z_2)$; $z_1,z_2\in\integer$. 
In particular, the four representatives are not conjugate inside the group, proving the assertion for $n=1$.
In the second case, manifesting when $n>1$,
\begin{equation*}
    \sigma=\begin{pmatrix}
        1 & 0 & 0 & 0 \\
        b & 0 & 1 & 0 \\
        -b & 1 & 0 & 0 \\
        0 & 0 & 0 & \mymathbb{1}_{n-2}
    \end{pmatrix}
\end{equation*}
and the restriction $b=0$ suffices, proving the assertion in any rank.
Collaterally, this last part of the argument proves the $n\geq3$ assertion for $W_n^\alpha=[W(A_{n-1})]^\alpha$.
\end{proof}
\begin{remark}
    When the linear part $G(d,1,n)$ allows for higher reflections, namely $d\in\{3,4,6\}$, powers of a reflection get all identified and there are just three orbits of hyperplanes for $n=1$ or four in higher rank \cite[Table III]{Malle1996}. The additional orbit is a byproduct of the moduli characterizing the non-genuine case, whose linear part is independent of the lattice (generic) parameter $\alpha$. E.g., in $[G(4,1,n)]^\alpha_1$ the above case $b=\alpha$, which translates to $b=i$, sits inside the $b=1$ orbit by conjugation of the element
    \begin{equation*}
    \begin{pmatrix}
        1 & 0 & 0 \\
        z_1+iz_2 & i & 0 \\
        0 & 0 & \mymathbb{1}_{n-1}
    \end{pmatrix},
\end{equation*}
whose linear part's explicit dependence on the lattice parameter $i$ is manifest.
\end{remark}

\section{The braid theorem}\label{sec:braid}

Time has come to prove the braid theorem for our new reflection presentations. En passant, the proofs establish the topological facet of the $x$-relations by specifying those generators in need of an explicit inversion.
In all cases, we first invoke a presentation of the braid group and then prove the isomorphism $\mathrm{Ar}(W_n^\alpha)\simeq\mathrm{Br}(W_n^\alpha)$ directly by checking both ways that the relations of one group suffice to satisfy those of the other group.  

\paragraph{Case $W^\alpha=[G(2,1,1)]^\alpha_1\simeq[W(A_1)]^\alpha$\\}
For the simplest group, $\mathrm{Tran}([G(2,1,1)]^\alpha_1)$ is just given by translations from the elementary lattice $\Lambda^\alpha=\integer+\alpha\integer\subset\complex$, while the linear part reads simply as $\langle-1\rangle$.
It easily follows that the set of reflection hyperplanes is the lattice $\frac{1}{2}\Lambda^\alpha$. Thus, the quotient $M_{[G(2,1,1)]^\alpha_1}/\mathrm{Tran}([G(2,1,1)]^\alpha_1)$ is homeomorphic to $\Sigma_{1,4}$, the four-punctured torus. With a further quotient by the linear part $\{\pm1\}$, no punctures get identified and a simple argument on invariants \cite[Appendix J]{Dubrovin1996} turns the torus into the sphere. In other words,
\begin{equation}\label{braidn=1}
    \mathrm{Br}([G(2,1,1)]^\alpha_1) \ \simeq \ \pi_1(\Sigma_{0,4})
\end{equation}
is nothing but the free group in three generators, manifestly matching the Artin group of the reflection presentation.
\begin{remark}
    Presentation \eqref{A1-pres} appeared more than thirty years ago \cite[Theorem 3.7 (I)]{Tak1994}, when Takebayashi derived presentations for the elliptic Weyl groups $W(X^{(1,1)})$, $X=A,B,C,D$. Beyond the simplest group, presentations \eqref{C-pres} and \eqref{A-pres} (respectively $X=C$ and $X=A$) crucially differ in order to achieve the braid theorem.   
\end{remark}

\paragraph{Case $W_n^\alpha=[G(2,1,n)]^\alpha_1,\ n\geq2$\\}
In higher rank, the same argument of \cite[Theorem 7.1]{Malle1996} extends interpretation \eqref{braidn=1} via the language of configuration spaces:
\begin{equation*}
    \mathrm{Br}([G(2,1,n)]^\alpha_1) \ \simeq \ \pi_1(\mathrm{UC}_n(\Sigma_{0,4})),
\end{equation*}
for the unordered configuration space of $n$ points on the four-punctured sphere
\begin{equation*}
    \mathrm{UC}_n(\Sigma_{0,4}):=\{(z_1,\cdots,z_n)\in\Sigma_{0,4}\ |\ z_i \neq z_j \text{ for }i\neq j\}\,/\,\mathfrak{S}_n.
\end{equation*}
Crucially, the fundamental group of such space admits the presentation
\begin{equation}\label{p1sigma04}
    \pi_1(\mathrm{UC}_n(\Sigma_{0,4})) \ = \ \left\langle\ \begin{matrix}
        u_1,u_2,u_3,u_4,\\t_1,\cdots,t_{n-1}\hfill
    \end{matrix}\ \left|\ \begin{array}{c}
        u_1u_2u_3u_4t_1 \cdots t_{n-1}t_{n-1}\cdots t_1=1,\hfill\\
        t_it_j=t_jt_i \text{\hspace{1.5em}for\hspace{.5em}} i-j>1,\hfill\\
        t_it_{i+1}t_i=t_{i+1}t_it_{i+1}\text{\hspace{1.5em}for\hspace{.5em}} i\leq n-2,\hfill\\
        u_it_j=t_ju_i  \text{\hspace{1.5em}for\hspace{.5em}}  j\geq2,\hfill\\
        u_it_1u_it_1=t_1u_it_1u_i,\hfill\\
        u_it^{-1}_1u_jt_1=t^{-1}_1u_jt_1u_i  \text{\hspace{1.5em}for\hspace{.5em}}  i<j\hfill      
    \end{array}\right.\right\rangle,
\end{equation}
obtained by adding the first relation\footnote{In order to later match the definition \cite{EGO2006} of higher rank GDAHAs, we also replaced each relation in the last group with an inverse of its conjugation.} (closedness) to the presentation for the fundamental group of the four-punctured plane's unordered configuration space found in \cite[§5]{Lambropoulou2000}---there denoted in the jargon of strands as $B_{4,n}$.
\begin{lemma}\label{lem:braidfourzero}
    Let $\mathbf{s}:=s_2\cdots s_n$. Then, the following assignments
    \begin{equation}\begin{matrix}
        u_1 & \mapsto & (s_1\cdots s_{n+2}s_n\cdots s_2)^{-1},\hfill\\
        u_2 & \mapsto & s_1,\hfill\\
        u_3 & \mapsto & \mathbf{s}\,s_{n+1}\mathbf{s}^{-1},\hfill\\
        u_4 & \mapsto & \mathbf{s}\,s_{n+2}\mathbf{s}^{-1},\hfill\\
        t_i & \mapsto & s_{i+1},\hfill
    \end{matrix}
    \end{equation}
extend to an isomorphism
\begin{equation}
\pi_1(\mathrm{UC}_n(\Sigma_{0,4})) \ \simeq \ \mathrm{Ar}([G(2,1,n)]^\alpha_1).
\end{equation}
\end{lemma}
\begin{proof}
We start with a preliminary explanation for the notation: a pair of round brackets triggers the corresponding braid relation or $x$-relation, while a pair of square brackets delimits the insertion of an identity via inverse elements.

Setting $\mathbf{s}^T:=s_n\cdots s_2$, the first relation follows as the image of $u_1$ is nothing but the element
$$(u_2u_3u_4\mathbf{s}\mathbf{s}^T)^{-1}.$$
The second and third groups of relations, corresponding to the Coxeter diagram $A_{n-1}$, are obvious too. For the remaining ones, we give again the basic manipulations for the $n=3$ case, as they easily adapt with increasing rank:
\begin{align*}
    u_1t_2&=s^{-1}_2s^{-1}_3s^{-1}_5s^{-1}_4(s^{-1}_3s^{-1}_2s_3)s^{-1}_1=(s^{-1}_2s^{-1}_3s_2)s^{-1}_5s^{-1}_4s^{-1}_3s^{-1}_2s^{-1}_1=t_2u_1,\allowdisplaybreaks\\[.5em]
    u_2t_2&=s_3s_1=t_2u_2,\allowdisplaybreaks\\[.5em]
    u_3t_2&=s_2s_3s_4(s^{-1}_3s^{-1}_2s_3)=(s_2s_3s_2)s_4s^{-1}_3s^{-1}_2=t_2u_3,\allowdisplaybreaks\\[.5em]
    u_4t_2&=s_2s_3s_5(s^{-1}_3s^{-1}_2s_3)=(s_2s_3s_2)s_5s^{-1}_3s^{-1}_2=t_2u_4,\allowdisplaybreaks\\[.5em]
 u_1t_1u_1t_1&=s^{-1}_2s^{-1}_3s^{-1}_5s^{-1}_4(s^{-1}_3s^{-1}_2 s^{-1}_3)s^{-1}_5s^{-1}_4s^{-1}_3(s^{-1}_1s^{-1}_2s^{-1}_1s_2)=\\&=(s^{-1}_2s^{-1}_3s^{-1}_2)s^{-1}_5s^{-1}_4s^{-1}_3s^{-1}_5s^{-1}_4(s^{-1}_2s^{-1}_3s_2)s^{-1}_1s^{-1}_2s^{-1}_1=\\&=s^{-1}_3s^{-1}_2s^{-1}_3s^{-1}_5(s^{-1}_4s^{-1}_3s^{-1}_5[s_3)s^{-1}_3]s^{-1}_4s_3s^{-1}_2s^{-1}_1s^{-1}_3s^{-1}_2s^{-1}_1=\\&=s^{-1}_3s^{-1}_2(s^{-1}_3s^{-1}_5s^{-1}_3s^{-1}_5)s_3s^{-1}_4s^{-1}_3s^{-1}_4s_3s^{-1}_2s^{-1}_1s^{-1}_3s^{-1}_2s^{-1}_1=\\&=s^{-1}_3s^{-1}_2s^{-1}_5s^{-1}_3s^{-1}_5s_3(s^{-1}_3s^{-1}_4s^{-1}_3s^{-1}_4)s_3s^{-1}_2s^{-1}_1s^{-1}_3s^{-1}_2s^{-1}_1=\\&=s^{-1}_3s^{-1}_2s^{-1}_5(s^{-1}_3s^{-1}_5s_3s^{-1}_4)s^{-1}_3s^{-1}_4s^{-1}_2s^{-1}_1s^{-1}_3s^{-1}_2s^{-1}_1=\\&=s^{-1}_3s^{-1}_5s^{-1}_4(s^{-1}_2s^{-1}_3s^{-1}_2)s^{-1}_1s^{-1}_5s^{-1}_4s^{-1}_3s^{-1}_2s^{-1}_1=t_1u_1t_1u_1,\allowdisplaybreaks\\[.5em]
 u_2t_1u_2t_1&=s_1s_2s_1s_2=s_2s_1s_2s_1=t_1u_2t_1u_2,\allowdisplaybreaks\\[.5em]
u_3t_1u_3t_1&=s_2s_3s_4s^{-1}_3s_2s_3s_4s^{-1}_3=s_2s_3s_4(s_2s_3s^{-1}_2)s_4s^{-1}_3=s_2s_3s_2(s_3s_4s_3s_4)s^{-1}_3s^{-1}_2s^{-1}_3=\\&=s_2(s_2s_3s_2)s_4s_3s_4(s^{-1}_2s^{-1}_3s^{-1}_2)=s_2s_2s_3s_4(s^{-1}_3s_2s_3)s_4s^{-1}_3s^{-1}_2=t_1u_2t_1u_2,\allowdisplaybreaks\\[.5em]
u_4t_1u_4t_1&=s_2s_3s_5s^{-1}_3s_2s_3s_5s^{-1}_3=s_2s_3s_5(s_2s_3s^{-1}_2)s_5s^{-1}_3=s_2s_3s_2(s_3s_5s_3s_5)s^{-1}_3s^{-1}_2s^{-1}_3=\\&=s_2(s_2s_3s_2)s_5s_3s_5(s^{-1}_2s^{-1}_3s^{-1}_2)=s_2s_2s_3s_5(s^{-1}_3s_2s_3)s_5s^{-1}_3s^{-1}_2=t_1u_4t_1u_4,\allowdisplaybreaks\\[.5em]
u_1t_1^{-1}u_2t_1&=s^{-1}_2s^{-1}_3s^{-1}_5s^{-1}_4s^{-1}_3s^{-1}_2(s^{-1}_1s^{-1}_2s_1s_2)=s^{-1}_2s_1s^{-1}_3s^{-1}_5s^{-1}_4s^{-1}_3s^{-1}_2s^{-1}_1=t_1^{-1}u_2t_1u_1,\allowdisplaybreaks\\[.5em]
u_1t_1^{-1}u_3t_1&=s^{-1}_2s^{-1}_3s^{-1}_5s^{-1}_4(s^{-1}_3s^{-1}_2s_3)s_4s^{-1}_3s^{-1}_1=(s^{-1}_2s^{-1}_3s_2)s^{-1}_5(s^{-1}_4s^{-1}_3s_4[s_3)s^{-1}_3]s^{-1}_2s^{-1}_3s^{-1}_1=\\&=s_3s^{-1}_2(s^{-1}_3s^{-1}_5s_3s_4)s^{-1}_3s^{-1}_4(s^{-1}_3s^{-1}_2s^{-1}_3)s^{-1}_1=s_3s_4(s^{-1}_2s^{-1}_3s^{-1}_2)s^{-1}_5s^{-1}_4s^{-1}_3s^{-1}_2s^{-1}_1=t_1^{-1}u_3t_1u_1,\allowdisplaybreaks\\[.5em]
u_1t_1^{-1}u_4t_1&=s^{-1}_2s^{-1}_3s^{-1}_5s^{-1}_4(s^{-1}_3s^{-1}_2s_3)s_5s^{-1}_3s^{-1}_1=(s^{-1}_2s^{-1}_3s_2)s^{-1}_5(s^{-1}_4s^{-1}_3s_5[s_3)s^{-1}_3]s^{-1}_2s^{-1}_3s^{-1}_1=\\&=s_3s^{-1}_2(s^{-1}_3s^{-1}_5s^{-1}_3s_5)s_3s^{-1}_4(s^{-1}_3s^{-1}_2s^{-1}_3)s^{-1}_1=s_3s_5(s^{-1}_2s^{-1}_3s^{-1}_2)s^{-1}_5s^{-1}_4s^{-1}_3s^{-1}_2s^{-1}_1=t_1^{-1}u_4t_1u_1,\allowdisplaybreaks\\[.5em]
u_2t_1^{-1}u_3t_1&=s_3s_4s^{-1}_3s_1=t_1^{-1}u_3t_1u_2,\allowdisplaybreaks\\[.5em]
u_2t_1^{-1}u_4t_1&=s_3s_5s^{-1}_3s_1=t_1^{-1}u_4t_1u_2,\allowdisplaybreaks\\[.5em]
u_3t_1^{-1}u_4t_1&=s_2s_3s_4(s^{-1}_3s^{-1}_2s_3)s_5s^{-1}_3=(s_2s_3s_2)s_4s^{-1}_3s_5s^{-1}_2s^{-1}_3=s_3s_2(s_3s_4s^{-1}_3s_5)s^{-1}_2s^{-1}_3=\\&=s_3s_5s_2s_3s_4(s^{-1}_3s^{-1}_2s^{-1}_3)=s_3s_5(s_2s_3s^{-1}_2)s_4s^{-1}_3s^{-1}_2=t_1^{-1}u_4t_1u_3.
\end{align*}
We conclude by repeating these $n=3$ checks for the inverse assignments
\begin{equation}\begin{matrix}
        s_1 & \mapsto & u_2,\hfill\\
        s_i & \mapsto & t_{i-1}, \quad 2\leq i \leq n,\\
        s_{n+1} & \mapsto & \mathbf{t}^{-1}u_3\mathbf{t},\hfill\\
        s_{n+2} & \mapsto & \mathbf{t}^{-1}u_4\mathbf{t},\hfill
    \end{matrix}
    \end{equation}
where now $\mathbf{t}:=t_1\cdots t_{n-1}$:
\begin{align*}
    s_1s_2s_1s_2&=u_2t_1u_2t_1=t_1u_2t_1u_2=s_2s_1s_2s_1,\allowdisplaybreaks\\[.5em]
    s_1s_3&=t_2u_2=s_3s_1,\allowdisplaybreaks\\[.5em]    
    s_1s_4&=t_2^{-1}(u_2t_1^{-1}u_3t_1)t_2=s_3s_1,\allowdisplaybreaks\\[.5em]
    s_1s_5&=t_2^{-1}(u_2t_1^{-1}u_4t_1)t_2=s_3s_1,\allowdisplaybreaks\\[.5em]
    s_2s_3s_2&=t_1t_2t_1=t_2t_1t_2=s_3s_2s_3,\allowdisplaybreaks\\[.5em]
    s_2s_4&=(t_1\cdot t_2^{-1}t_1^{-1})u_3t_1t_2=t_2^{-1}t_1^{-1}u_3(t_2t_1t_2)=s_4s_2,\allowdisplaybreaks\\[.5em]
    s_2s_5&=(t_1\cdot t_2^{-1}t_1^{-1})u_4t_1t_2=t_2^{-1}t_1^{-1}u_4(t_2t_1t_2)=s_5s_2,\allowdisplaybreaks\\[.5em]
    s_3s_4s_3s_4&=t_1^{-1}u_3(t_2^{-1}t_1t_2)u_3t_1t_2=t_1^{-1}t_2^{-1}t_1^{-1}(u_3t_1u_3t_1)t_2t_1t_2=\\&=(t_2^{-1}t_1^{-1}t_2^{-1})u_3t_1u_3(t_2t_1t_2)t_2=t_2^{-1}t_1^{-1}u_3(t_1t_2t_1^{-1})u_3t_1t_2t_2=s_4s_3s_4s_3,\allowdisplaybreaks\\[.5em]    s_3s_5s_3s_5&=t_1^{-1}u_4(t_2^{-1}t_1t_2)u_4t_1t_2=t_1^{-1}t_2^{-1}t_1^{-1}(u_4t_1u_4t_1)t_2t_1t_2=\\&=(t_2^{-1}t_1^{-1}t_2^{-1})u_4t_1u_4(t_2t_1t_2)t_2=t_2^{-1}t_1^{-1}u_4(t_1t_2t_1^{-1})u_4t_1t_2t_2=s_5s_3s_5s_3,\allowdisplaybreaks\\[.5em]    s_3s_4s_3^{-1}s_5&=t_1^{-1}u_3(t_1t_2^{-1}t_1^{-1})u_3t_1t_2=t_1^{-1}t_2^{-1}u_3t_1^{-1}u_3(t_2t_1t_2)=t_1^{-1}t_2^{-1}(u_3t_1^{-1}u_3t_1)t_2t_1=\\&=(t_1^{-1}t_2^{-1}t_1^{-1})u_3t_1t_2u_3t_1=t_2^{-1}t_1^{-1}u_3(t_2^{-1}t_1t_2)u_3t_1=s_5s_3s_4s_3^{-1}.\qedhere
\end{align*}
\end{proof}
\begin{figure}[t]
    \centering
    \includegraphics[height=12em]{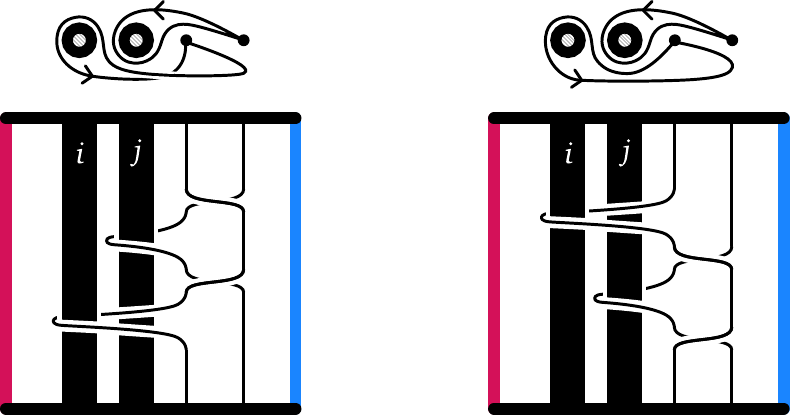}
    \caption{The $n=2$ visualization of the ``elliptic braid relation'' $[u_i,t_1^{-1}u_jt_1]=0,\ i<j$, under counterclockwise convention. E.g., the left pair of strands is read downwards exactly as $u_it_1^{-1}u_jt_1$. The equivalence of the two sides becomes manifest in the picture, even more so when seen from the top. In configuration terms, $t_i$ swaps strand $i$ over strand $i+1$ while $u_k$ loops the first strand counterclockwise around hole $k$.}
    \label{fig:ell-braid}
\end{figure}
\begin{remark}
The affinity of presentation \eqref{p1sigma04} with the language of Artin groups was already observed in \cite[Remark 4]{Lambropoulou2000}; there, for the general $B_{m,n}$ case, $(m-1)m/2$ extra relations are needed to complement the given Coxeter-type diagram. Despite including the further closedness relation, the above presentation for $\mathrm{Ar}([G(2,1,n)]^\alpha_1)$ fails a purely Artin description by the \emph{single} elliptic braid relation involving the pair of additional generators, in that the remaining five elliptic braid relations get encoded in purely Coxeter terms.     
\end{remark}

\paragraph{Case $W_n^\alpha=[W(A_{n-1})]^\alpha,\ n\geq3$\\}
In type $A$, we rely on Ion-Roller's investigation of ``special'' toric configurations, i.e., whose product of points gives the identity element of the torus $\mathbb{T}=\Sigma_{1,0}$. In particular, under minor\footnote{Essentially, replacing the Heisenberg manifold $\mathbb{H}$ with the torus $\mathbb{T}$.} adaptations of \cite[Theorem 6.6]{IR2025}, the following holds: 
\begin{proposition}[\cite{IR2025}]
    The regular orbit space for the elliptic Weyl group $W(A_{n-1}^{(1,1)})$ is homeomorphic to the space of special unordered configurations of $n$ points on the torus
    \begin{equation*}
    \mathrm{SUC}_n(\Sigma_{1,0})=\{(z_1,\cdots,z_n)\in\Sigma_{1,0}\ |\ z_1\cdots z_n=1 \text{, } z_i \neq z_j \ \mathrm{for} \ i\neq j\}\,/\,\mathfrak{S}_n.
\end{equation*}
\end{proposition}
In particular, since $W(A_{n-1}^{(1,1)})\simeq W(A_{n-1})\ltimes Q(A_{n-1})^{\times2}\simeq [W(A_{n-1})]^\alpha$, we interpret again in terms of configuration spaces as
\begin{equation*}
    \mathrm{Br}([W(A_{n-1})]^\alpha) \ \simeq \ \pi_1(\mathrm{SUC}_n(\Sigma_{1,0})),
\end{equation*}
unlocking the following presentation \cite[Theorem 6.8]{IR2025}:
\begin{equation}\label{pi1SUC}
    \pi_1(\mathrm{SUC}_n(\mathbb{T})) \ = \ \left\langle\ \begin{matrix}
        r_0,\cdots,r_{n-1},\\t_1,\cdots,t_{n-1}\hfill
    \end{matrix}\ \left|\ \begin{array}{c}
        r_ir_j=r_jr_i  \text{\hspace{1.5em}for\hspace{.5em}}  i-j\neq1 \mathrm{\ mod\ }n,\hfill\\
        r_ir_jr_i=r_jr_ir_j  \text{\hspace{1.5em}for\hspace{.5em}}  i-j=1 \mathrm{\ mod\ }n,\hfill\\
        t_it_j=t_jt_i,\hfill\\
        r_it_j=t_jr_i  \text{\hspace{1.5em}for\hspace{.5em}}  i-j\neq1 \mathrm{\ mod\ }n,\hfill\\
        r_it_{i+1}r_i=r_{i+1}t_ir_{i+1}=t_it_{i+1}  \text{\hspace{1.5em}for\hspace{.5em}} 1\leq i\leq n-2,\\
        r_0t_ir_0=t_i(t_1\cdots t_{n-1})^{-1}  \text{\hspace{1.5em}for\hspace{.5em}}  i=1,n-1\hfill    
    \end{array}\right.\right\rangle.
\end{equation}
\begin{lemma}\label{lem:braidonezero}
    The following assignments
    \begin{equation}\label{braidonezero}\begin{matrix}
        r_0 & \mapsto & s_{n},\hfill\\
        r_i & \mapsto & s_i,\quad 1\leq i \leq n-1,\hfill\\
        t_1 & \mapsto & (s_{n-1} \cdots s_2)^{-1}s_{n+1}s_{n-1}\cdots s_1,\hfill\\
        t_i & \mapsto & (s_1 \cdots s_{i-1}s_{n-1} \cdots s_{i+1})^{-1}s_{n+1}s_1 \cdots s_{i-1}s_{n-1} \cdots s_i, \quad 2\leq i \leq n-2,\hfill\\
        t_{n-1} & \mapsto & (s_1 \cdots s_{n-2})^{-1}s_{n+1}s_1 \cdots s_{n-1},\hfill
    \end{matrix}
    \end{equation}
extend to an isomorphism
\begin{equation}
\pi_1(\mathrm{SUC}_n(\Sigma_{1,0})) \ \simeq \ \mathrm{Ar}([W(A_{n-1})]^\alpha).
\end{equation}
\end{lemma}
\begin{proof}
We stick with our notation: round brackets trigger their braid/$x$- relation while square brackets delimit the insertion of an identity.
The first two groups of relations, corresponding to the Coxeter diagram $A_{n-1}^{(1)}$, are obvious. For the remaining ones, we give again basic manipulations for the $n=4$ case involving all core types of relations, as they easily adapt with increasing rank:
\begin{align*}
    t_1t_2&=s^{-1}_2s^{-1}_3s_5(s_3s_2s^{-1}_3)s_5s_1s_3s_2=(s^{-1}_2s^{-1}_3s^{-1}_2)s_3s_5s_3s_2s_3s_1s_2=s^{-1}_3s_5s_3(s_2s_1s_2)=\\&=s^{-1}_3(s_5s_1[s^{-1}_5)s_3s^{-1}_3s_5]s_3s_2s_1=s^{-1}_3s^{-1}_1s_5s_1s_3[s_2s^{-1}_2]s^{-1}_3s_5s_3s_2s_1=t_2t_1,\allowdisplaybreaks\\[.5em]
    t_1t_3&=s^{-1}_2s^{-1}_3s_5s_3(s_2s_1\cdot s^{-1}_2)s^{-1}_1s_5s_1s_2s_3=s^{-1}_2(s^{-1}_3s_5s_3)s^{-1}_1s_5s_2s_1s_2s_3=s^{-1}_2s_5s_3(s^{-1}_5s^{-1}_1s_5)s_2s_1s_2s_3=\\&=s^{-1}_2s_5s_1s_3s^{-1}_5(s^{-1}_1s_2s_1)s_2s_3=s^{-1}_2s_5s_1s_3s^{-1}_5(s_2s_3[s^{-1}_2)s_2]s_1=s^{-1}_2s_5s_1(s_3s^{-1}_5s^{-1}_3)s_2s_3s_2s_1=\\&=s^{-1}_2(s_5s_1s^{-1}_5)s^{-1}_3s_2s_5s_3s_2s_1=s^{-1}_2s^{-1}_1s_5s_1(s^{-1}_3s_2[s_3)s^{-1}_3]s_5s_3s_2s_1=t_3t_1,\allowdisplaybreaks\\[.5em]    t_2t_3&=s^{-1}_1(s^{-1}_3s_5s_3)s_5s_1s_2s_3=s^{-1}_1s_5s_1(s_3s_2s_3)=s^{-1}_1s_5(s_1s_2[s^{-1}_1)s_1]s_3s_2=\\&=(s^{-1}_1s^{-1}_2[s^{-1}_1)(s_1]s_5s_1)s_2s_3s_1s_2=s^{-1}_2s^{-1}_1s_5(s^{-1}_2s_1s_2)s_5s_3s_1s_2=t_3t_2,\\[-3.5em]
    \end{align*}
    \begin{align*}
    r_1t_2r_1&=s^{-1}_3s_5s_3(s_1s_2s_1)=r_2t_1r_2=t_1t_2,\allowdisplaybreaks\\[.5em]
    r_2t_3r_2&=s^{-1}_1s_5s_1(s_2s_3s_2)=r_3t_2r_3=t_2t_3,\allowdisplaybreaks\\[.5em] r_0t_1r_0&=s_4s^{-1}_2(s^{-1}_3s_5s_3s_2s_1s_4)=s^{-1}_2s^{-1}_3s^{-1}_2s^{-1}_1s^{-1}_5s_1=(t_2t_3)^{-1},\allowdisplaybreaks\\[.5em]    r_0t_3r_0&=s_4s^{-1}_2(s^{-1}_1s_5s_1s_2s_3s_4)=s^{-1}_2s^{-1}_1s^{-1}_2s^{-1}_3s^{-1}_5s_3=(t_1t_2)^{-1},\allowdisplaybreaks\\[.5em]
    r_1t_3&=(s_1s^{-1}_2s^{-1}_1)s_5s_1s_2s_3=s^{-1}_2s^{-1}_1s_5(s_2s_1s_2)s_3=s^{-1}_2s^{-1}_1s_5s_1s_2s_3s_1=t_3r_1,\allowdisplaybreaks\\[.5em]
    r_3t_1&=(s_3s^{-1}_2s^{-1}_3)s_5s_3s_2s_1=s^{-1}_2s^{-1}_3s_5(s_2s_3s_2)s_1=s^{-1}_2s^{-1}_3s_5s_3s_2s_1s_3=t_1r_3,\allowdisplaybreaks\\[.5em]
    r_0t_2&=s_4s^{-1}_1(s^{-1}_3s_5s_3)s_1s_2=(s_4s^{-1}_1s_5)s_3s^{-1}_5s_1s_2=s^{-1}_1s^{-1}_2s^{-1}_3s^{-1}_5s_3s^{-1}_4(s^{-1}_3s^{-1}_2s_3)s^{-1}_1s^{-1}_5s_1s_2=\\&=s^{-1}_1s^{-1}_2s^{-1}_3s^{-1}_5s_3s_2(s^{-1}_4s^{-1}_3s^{-1}_2s^{-1}_1s^{-1}_5s_1)s_2=s^{-1}_1s^{-1}_2s^{-1}_3s^{-1}_5(s_3s_2s^{-1}_3)s_5s_3s_2s_1s_4s_2=\\&=s^{-1}_1(s^{-1}_2s^{-1}_3s^{-1}_2)(s^{-1}_5s_3s_5)s_2s_3s_2s_1s_2s_4=s^{-1}_1s^{-1}_3s_5(s^{-1}_2s^{-1}_3s_2)s_3s_2s_1s_2s_4=t_2r_0,
\end{align*}
with the last commutation relying on the ``inhomogeneous form'' $$s_4s^{-1}_1s_5=s^{-1}_1s^{-1}_2s^{-1}_3s^{-1}_5s_3s^{-1}_4s^{-1}_3s^{-1}_2s^{-1}_1$$
of the $x$-relation.

We finish by repeating these $n=4$ checks for the inverse assignments
\begin{equation}\begin{matrix}
        s_i & \mapsto & r_{i}, \quad i \leq 3,\hfill\\
        s_4 & \mapsto & r_0,\hfill\\
        s_5 & \mapsto & r_3r_2t_1r^{-1}_1r^{-1}_2r^{-1}_3,\hfill
    \end{matrix}
    \end{equation}
where we chose to write $s_5$, among the three equivalent (leftmost) expressions
\begin{align*}
    &r_1r_2t_3r^{-1}_3r^{-1}_2r^{-1}_1=r_1(r_2t_3)r^{-1}_3r^{-1}_2r^{-1}_1=r_1t_2t_3(r^{-1}_2r^{-1}_3r^{-1}_2)r^{-1}_1=r_1(t_2t_3r^{-1}_3)r^{-1}_2r^{-1}_3r^{-1}_1=\\=&r_1r_3t_2r^{-1}_2r^{-1}_3r^{-1}_1=r_3(r_1t_2)r^{-1}_2r^{-1}_1r^{-1}_3=r_3t_1t_2(r^{-1}_1r^{-1}_2r^{-1}_1)r^{-1}_3=r_3(t_1t_2r^{-1}_2)r^{-1}_1r^{-1}_2r^{-1}_3=\\=&r_3r_2t_1r^{-1}_1r^{-1}_2r^{-1}_3,
\end{align*}
 with respect to $t_1$:
\begin{align*}
    s_5s_1s_5&=r_3r_2t_1(r^{-1}_1r^{-1}_2)r_2t_1r^{-1}_1r^{-1}_2r^{-1}_3=r_3(r_2t_1r_2)r^{-1}_1t_1r^{-1}_1r^{-1}_2r^{-1}_3=r_3r_1(t_2t_1r^{-1}_1)r^{-1}_2r^{-1}_3=\\&=r_1r_3r_1t_2r^{-1}_2r^{-1}_3=r_1 \cdot r_1r_3t_2r^{-1}_2r^{-1}_3[r^{-1}_3 \cdot r_1]=s_1s_5s_1,\allowdisplaybreaks\\[.5em]
    s_2s_5&=(r_2r_3r_2)t_1r^{-1}_1r^{-1}_2r^{-1}_3=r_3r_2t_1r^{-1}_1(r_3r^{-1}_2r^{-1}_3)=r_3r_2t_1r^{-1}_1r^{-1}_2r^{-1}_3r_2=s_5s_2,\allowdisplaybreaks\\[.5em]
    s_5s_3s_5&=r_1r_2t_3r^{-1}_3r^{-1}_2r^{-1}_1 \cdot r_3 \cdot r_1r_2t_3r^{-1}_3r^{-1}_2r^{-1}_1=r_1r_2t_3(r^{-1}_3r^{-1}_2r_3)r_2t_3r^{-1}_3r^{-1}_2r^{-1}_1=\\&=r_1(r_2t_3r_2)r^{-1}_3t_3r^{-1}_3r^{-1}_2r^{-1}_1=r_3r_1(t_2t_3)r^{-1}_3r^{-1}_2r^{-1}_1=r_3r_1r_2t_3(r_2r^{-1}_3r^{-1}_2)r^{-1}_1=\\&=r_3\cdot r_1r_2t_3r^{-1}_3r^{-1}_2r^{-1}_1\cdot r_3=s_3s_5s_3,
\end{align*}\vspace{-3em}
\begin{align*}
s_3s^{-1}_4s^{-1}_3s^{-1}_2s^{-1}_1s^{-1}_5&=r_3r^{-1}_0(r^{-1}_3r^{-1}_2r_3)r^{-1}_1r_2r_1t^{-1}_1r^{-1}_2r^{-1}_3=r_3r^{-1}_0r_2r^{-1}_3r^{-1}_2(r^{-1}_1r_2r_1)t^{-1}_1r^{-1}_2r^{-1}_3\\&=r_3r^{-1}_0r_2r^{-1}_3r_1(r^{-1}_2t^{-1}_1r^{-1}_2)r^{-1}_3=r_3r_2r^{-1}_0r_1t^{-1}_1(r^{-1}_3t^{-1}_2r^{-1}_3)=\\&=r_3r_2r^{-1}_0r_1(t^{-1}_1t^{-1}_2)t^{-1}_3=r_3r_2(r^{-1}_0t^{-1}_2t^{-1}_3)r^{-1}_1=\\&=r_3r_2t_1[r^{-1}_1r^{-1}_2r^{-1}_3r_3r_2r_1]r_0r^{-1}_1=s_5s_3s_2s_1s_4s^{-1}_1.
\end{align*}
Notice that, to check $s_5s_3s_5=s_3s_5s_3$, it was more convenient to express $s_5$ with respect to $t_3$. 
\end{proof}
The proof of \Cref{thm:second} is thus completed.
\begin{remark}
    Thanks to formulae \eqref{braidonezero}, we can now get to grips with the mysterious $x$-relation \eqref{braid2}. For special configurations, translating along all lattice generators amounts to the identity:
    \begin{equation}\label{unital}
        t_0 \cdots t_{n-1}=1.
    \end{equation}
    Therefore, reducing the generating set ends up modifying a pair of van der Lek's ``pushrelations'' \cite[Remark 2.8]{vanderLek1983} in form $t_jr_i=r_i^{-1}(\ldots)$. It is then easy to check that the $x$-relation plays the role of a reflection codification for such modified relations. For $n=3$, a visual take on the unital condition \eqref{unital} and the $x$-relation is respectively given in Figures \ref{fig:special} and \ref{fig:braid2}.  
\end{remark}
\begin{figure}[t!]
    \centering
    \includegraphics[width=30em]{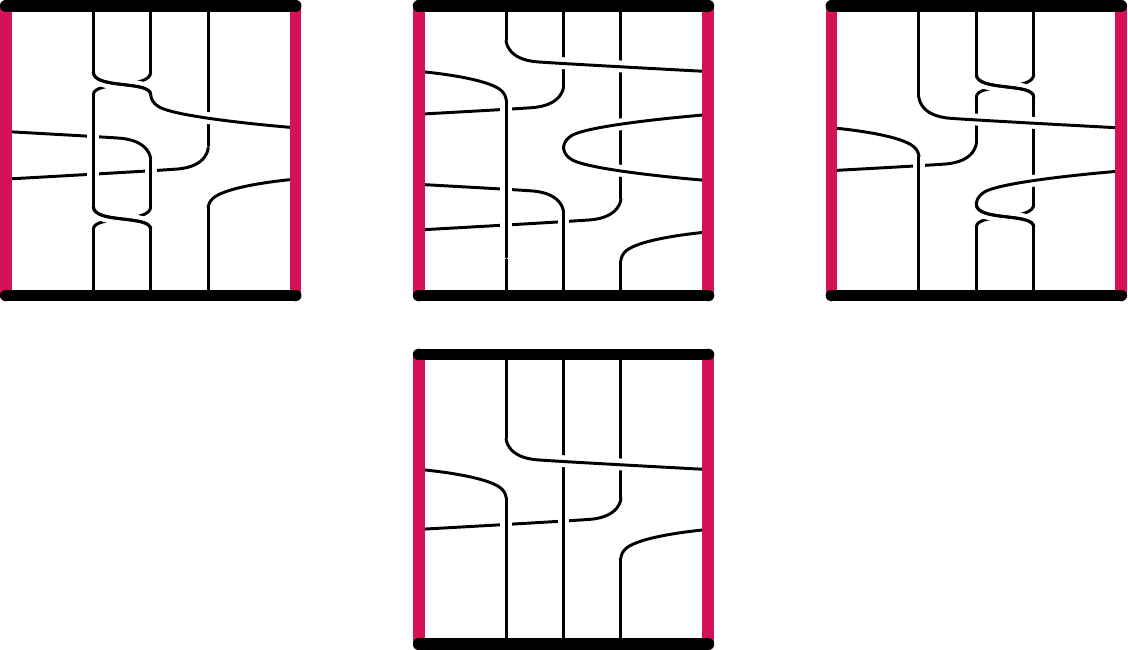}
    \caption{The $n=3$ visualization of the pushrelation $r_1t_2r_1=t_2t_1=r_2t_1r_2$, together with a meridional side view (bottom) simplifying $t_2t_1$ into the translation $\alpha e_1-\alpha e_3$ as expected by the special configuration condition. In configuration terms, $r_i$ swaps strand $i$ over strand $i+1$ while $t_j$, corresponding to the vector $\alpha e_j-\alpha e_{j+1}$, loops along the longitude strands $j$ and $j+1$ in opposite directions---with strand crossing ruled by the points' diagonal arrangement in the Conventions. E.g., $t_2$ loops the second point rightward which, from the meridional side view, is seen crossing over the third point but under the first.}
    \label{fig:special}
\end{figure}
\begin{figure}[t!]
    \centering
    \includegraphics[height=8em]{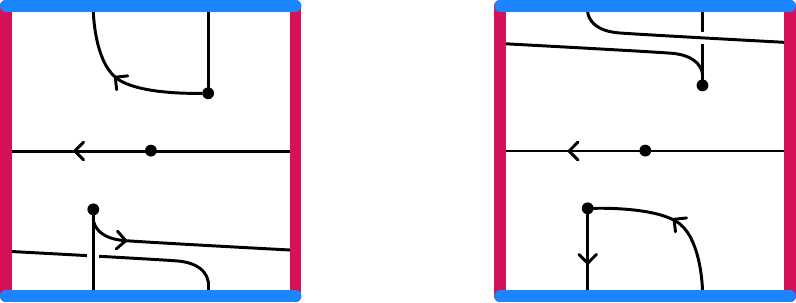}
    \caption{The $n=3$ visualization of the ``special pushrelation'' $r_0t_1=t_2^{-1}r_0^{-1}$, which is best seen from the top. In configuration terms, $r_0$ swaps the last strand over the first while looping them along the \emph{meridian} (as translations now correspond to the first copy of the lattice) in opposite directions.}
    \label{fig:braid2}
\end{figure}
In fact, we independently discovered formulae \eqref{pi1SUC} by central evaluation ($t_0^{-1}=t_1\cdots t_{n-1}$) of van der Lek's presentation \cite[(2.12)]{vanderLek1983} for the extended Artin group ${\mathrm{\tilde{A}r}}(A_{n-1}^{(1)})$. As anticipated in \Cref{rmk:conj}, this abstract group is isomorphic to the braid group of the hyperbolic extension $\tilde{W}(A^{(1,1)}_n)$.
Therefore, we find ourselves with a proof in type $A$ of an appealing fact: the signature topological feature of van der Lek's presentation persists after central evaluation, namely setting the central generator to the unit delivers a presentation for the braid group of the elliptic Weyl group. To the best of our knowledge, no conceptual proof of this phenomenon is available, motivating our formulation of the following  
\begin{conjecture}\label{conj:vdL}
    Let $R$ be any elliptic root system, with underlying affine part $R_a$, and denote by $\tilde{c}$ the hyperbolic Coxeter element of $\tilde{W}(R)$ such that $\tilde{W}(R)/\langle\tilde{c}\rangle \simeq W(R)$ \cite{ST1997}. 
    Then,
    \begin{equation}
        \mathrm{\tilde{A}r}(R_a)\,/\,\langle\tilde{c}\rangle\ \simeq\ \mathrm{Br}(W(R)).
    \end{equation}
\end{conjecture}
\begin{remark}\label{rmk:IonSahi2}
    Combining \Cref{rmk:IonSahi1} with \Cref{lem:braidfourzero}, we further support the conjecture with a positive answer also in type $C$, namely $R=C^{(1,1)}$. Indeed, $[G(2,1,n)]^\alpha_1 \simeq W(C^{(1,1)}_n)$ and it is immediate to see that central evaluation reduces Ion-Sahi's presentation \cite[Proposition 5.6]{IS2020} to our presentation for $\mathrm{Ar}([G(2,1,n)]^\alpha_1)$, i.e., $\mathrm{\tilde{A}r}(C^{(1)}_n)/\langle\dutchcal{C}\rangle\simeq\mathrm{Br}(W(C^{(1,1)}_n))$.
\end{remark}

\section{Generic Hecke algebras}\label{sec:Hecke}

This final section is the culmination of our topological investigations: after defining the generic Hecke algebra, we first match the central evaluation of the expected EHA in type $A$, and then give GDAHA the promised genesis in crystallographic complex reflection terms.

Denoting by $\mathrm{Hyp}(W)$ the set of reflection hyperplanes associated with the reflection group $W$ and by $W_H$ the pointwise stabilizer of the hyperplane $H\in\mathrm{Hyp}(W)$, let
\begin{equation}
    \big\{ \ \mymathbb{s}_{[H],k} \ : \ [H]\in\mathrm{Hyp}(W)/W, \ 1\leq k \leq |W_H| \ \big\}
\end{equation}
be a collection of formal parameters ranging in both the conjugacy classes of hyperplanes and $1\leq k \leq e_{[H]}:=|W_H|$.
For a reflection $\sg{i}$ of order $e_i$ that fixes $H_i\in\mathrm{Hyp}(W)$, define $\{\mymathbb{s}_{ij}\}_{1\leq j \leq e_i}$ to be the corresponding subcollection of parameters, i.e., $\mymathbb{s}_{ij}:=\mymathbb{s}_{[H_i],j}$. In particular, reflections whose hyperplanes lie in the same orbit---which is always the case when conjugate in $W$---share the subcollections while, for $n\in\naturali$ such that $e_i/n\in\naturali$, the $n$-th power of a reflection $\sg{i}$ gets the smaller subcollection $\{\mymathbb{s}_{ik}\}_{1 \leq k \leq \nicefrac{e_i}{n}}$. 

Hereafter, let $W$ be a Steinberg crystallographic complex reflection group in an infinite family (with reflection presentation as in Tables \ref{tab:nongenuine} or \ref{tab:genuine}) and denote by $\sg{0}$ the reflection singled out by the group's extra relation of order $e_0\in\naturali$.
Supported by our algebraic braid theorems, we give the following
\begin{definition}\label{def:Hecke}
    The \emph{generic Hecke algebra} of $W$ is defined as the $\complex[\mymathbb{s}_{ij}^{\pm1}]$-algebra with generators $S_i$ in bijection with the nodes of the group's reflection presentation, which satisfy the relations of the Artin group $\mathrm{Ar}(W)$ together with the (generic) Hecke ones
    \begin{equation}\label{generic-Hecke}
        \prod_{1\leq j\leq e_i}(S_i-\mymathbb{s}_{ij})=0,\quad\hspace{.45em}\text{ for }\sg{i}^{e_i}=1,\ i\geq0,
    \end{equation}    
    where $S_0:=\sigma_0\hspace{-.1em}\raisebox{-.2ex}[-1ex][-1ex]{|}_{\sigma\leftarrow S}$ is obtained by replacing each reflection in $\sg{0}=\prod_k\sg{k}$ with the corresponding Hecke generator $S_k$.
\end{definition}
Notice that, as usual, under the ``cyclotomic'' specialization $\mymathbb{s}_{*j} \mapsto e^{2\pi ij/e_*}$ the generic Hecke algebra manifests as a deformation of the group algebra $\complex W$.

Before attacking GDAHA, for the sake of completeness in non-genuine deformations, let us explain the connection between $[W(A_{n-1})]^\alpha$ and the theory of elliptic Hecke algebras. Following a suggestion by Bogdan Ion, one can define a third additional generator for $\mathrm{Ar}([W(A_{n-1})]^\alpha)$ as
\begin{equation}
    s_{n+2}:=(s_{n}s_{n+1}s_{n-1}\cdots s_1\cdots s_{n-1})^{-1}=\mathbf{s}^{-1}s_0^{-1}\mathbf{s}, \quad \mathbf{s}:=s_1\cdots s_{n-1},
\end{equation}
and easily verify that, provided the evaluation $\dutchcal{C}=1$, the resulting presentation coincides with Ion-Sahi's triple dot group $\mathbf{B}(\dddot{A}_{\!n-1})$ \cite{IS2020}. In particular, the defining factors of the central element $\dutchcal{C}$ correspond as
\begin{equation}
    \Theta_{01} \ \leftrightarrow \ s_{n+2}, \qquad \Theta_{02} \ \leftrightarrow \ s_{n}, \qquad \Theta_{03} \ \leftrightarrow \ s_{n+1}, \qquad \Theta \ \leftrightarrow \ s_{n-1}\cdots s_1\cdots s_{n-1}. 
\end{equation}E.g., for $n=4$
\begin{equation*}
    s_2s_6=(s_2s^{-1}_3s^{-1}_2)s^{-1}_1s^{-1}_2s^{-1}_3s^{-1}_5s^{-1}_4=s^{-1}_3s^{-1}_2s^{-1}_1(s_3s^{-1}_2s^{-1}_3)s^{-1}_5s^{-1}_4=s_6s_2,
\end{equation*}
and the following identities show that the $x$-relation indeed makes $s_{n+2}$ turn the diagram to a \emph{triple} affine extension:
\begin{align*}    s_1s_6s_1&=s^{-1}_3s_1(s^{-1}_2s^{-1}_1s^{-1}_2)s^{-1}_3s^{-1}_5s^{-1}_4s_1=s^{-1}_3s^{-1}_2s^{-1}_1(s^{-1}_3[s^{-1}_2s^{-1}_3)s^{-1}_5s_5s_3s_2]s^{-1}_5(s^{-1}_4s_1[s_4)s^{-1}_4]=\\&=s^{-1}_3s^{-1}_2s^{-1}_1s^{-1}_2s^{-1}_3s^{-1}_5s^{-1}_2(s_5s_3s^{-1}_5)s_2s_1s_4s^{-1}_1s^{-1}_4=\\&=s^{-1}_3s^{-1}_2s^{-1}_1s^{-1}_2s^{-1}_3s^{-1}_5s^{-1}_2s^{-1}_3(s_5s_3s_2s_1s_4s^{-1}_1)s^{-1}_4=\\&=s^{-1}_3s^{-1}_2s^{-1}_1s^{-1}_2s^{-1}_3s^{-1}_5s^{-1}_4[s_1s_1^{-1}](s^{-1}_2s^{-1}_3s^{-1}_2)s^{-1}_1s^{-1}_5s^{-1}_4=\\&=s^{-1}_3s^{-1}_2s^{-1}_1s^{-1}_2s^{-1}_3s^{-1}_5s^{-1}_4s_1s^{-1}_3(s_1^{-1}s^{-1}_2s^{-1}_1)s^{-1}_3s^{-1}_5s^{-1}_4=s_6s_1s_6,\allowdisplaybreaks\\[.5em]
s_6s_3s_6&=s^{-1}_3(s^{-1}_2s^{-1}_1s^{-1}_2)s^{-1}_3s^{-1}_2s^{-1}_5(s^{-1}_4s^{-1}_1s^{-1}_2s^{-1}_3s^{-1}_5)s^{-1}_4=\\&=s^{-1}_1s^{-1}_3s^{-1}_2s^{-1}_3s^{-1}_1s^{-1}_2s^{-1}_5(s^{-1}_1s_5s_1)s_2(s_3s_4s^{-1}_3)s^{-1}_4=\\&=s^{-1}_1s^{-1}_3s^{-1}_2s^{-1}_3(s^{-1}_1s^{-1}_2s_1)s_2s^{-1}_5s^{-1}_4s_3=s^{-1}_1(s^{-1}_3s^{-1}_2s^{-1}_3)s_2s^{-1}_1s^{-1}_5s^{-1}_4s_3=\\&=[s_3s_3^{-1}](s^{-1}_1s^{-1}_2s^{-1}_1)s^{-1}_3s^{-1}_5s^{-1}_4s_3=s_3s_6s_3.
\end{align*}
Therefore, as a straightforward corollary of \cite[Theorem 11.5]{IS2020}, we get the following
\begin{theorem}
    Provided the specialization
    \begin{equation*}
            \mymathbb{s}_{i2}=:\mymathbb{s}_{2} \ \mapsto \ -\mymathbb{s}_{1}^{-1}:=-\mymathbb{s}_{i1}^{-1}, \quad 0\leq i \leq n+1,
    \end{equation*}
    the generic Hecke algebra of $[W(A_{n-1})]^\alpha$, $n\geq3$, is isomorphic to the DAHA of type $A_{n-1}^{(1)}$ at $q=1$.
\end{theorem}
Read in the language of \cite[Definition 3.2.3]{SS2009}, the theorem involves the ``small'' version of the DAHA, and can be thus equivalently stated as the anticipated match with the central evaluation of Saito-Shiota's EHA in type $A_{n-1}^{(1,1)}$.

\subsection{GDAHA revisited}

We are now finally ready to prove our third and last result \ref{thm:third}. 

\begin{definition}\label{def:GDAHA}
The \emph{generalized double affine Hecke algebra} of rank $n$, attached to the star-shaped simply laced Coxeter diagram $\pazocal{D}$ with $k=1,\ldots,m$ legs each of length $d_k$, is defined as
    \begin{equation}
    \complex[\mymathbb{u}_{kj}^{\pm1},\mymathbb{t}^{\pm1}]\left\langle\ \begin{matrix}
        U_1,\cdots,U_m,\hfill\\T_1,\cdots,T_{n-1}
    \end{matrix}\ \left|\ \begin{array}{c}
        \prod_{j=1}^{d_k}(U_k-\mymathbb{u}_{kj})=(T_i-\mymathbb{t})(T_i+\mymathbb{t}^{-1})=0,\hfill\\
        U_1\cdots U_mT_1 \cdots T_{n-1}T_{n-1}\cdots T_1=1,\hfill\\
        T_iT_j=T_jT_i  \text{\hspace{1.5em}for\hspace{.5em}} i-j>1,\hfill\\
        T_iT_{i+1}T_i=T_{i+1}T_iT_{i+1}\text{\hspace{1.5em}for\hspace{.5em}} i \leq n-2,\hfill\\
        U_iT_j=T_jU_i \text{\hspace{1.5em}for\hspace{.5em}} j\geq2,\hfill\\
        U_iT_1U_iT_1=T_1U_iT_1U_i,\hfill\\
        U_iT^{-1}_1U_jT_1=T^{-1}_1U_jT_1U_i\text{\hspace{1.5em}for\hspace{.5em}}i<j\hfill

    \end{array}\right.\right\rangle.
\end{equation}
\end{definition}
\begin{theorem}\label{thm:GDAHA}
    Up to parameter specialization, the generic Hecke algebras for the crystallographic complex reflection groups $[G(2,1,n)]^\alpha_1$, $[G(3,1,n)]_1$, $[G(4,1,n)]_{1}$, and $[G(6,1,n)]$ are isomorphic to the rank $n$ (affine Dynkin) GDAHAs, respectively for $\pazocal{D}=D_4^{(1)},E_6^{(1)},E_7^{(1)},$ and $E_8^{(1)}$.
\end{theorem}
\begin{proof}
    The first match follows from ``capitalization'' of \Cref{lem:braidfourzero}, as the rank $n$ GDAHA of type $D_4^{(1)}$, itself nothing but Sahi's DAHA of type $(\check{C}_n,C_n)$ \cite[Proposition 3.3.2]{EGO2006} and thus equivalently Saito-Shiota's EHA of type $C_n^{(1,1)}$ \cite{SS2009}, deforms the braid group of the four-punctured sphere. Deformation \ref{def:Hecke} is indeed the most generic, and recovers the one defining the GDAHA via the specialization
    \begin{equation}\label{specialization}
        \begin{matrix}
            \mymathbb{s}_{i1} \ \mapsto \ \mymathbb{t}, & \mymathbb{s}_{i2} \ \mapsto \ -\mymathbb{t}^{-1}, & 2\leq i \leq n,
        \end{matrix}
    \end{equation}
    and identifications
        $\mymathbb{s}_{0j} \mapsto \mymathbb{u}_{1j}^{-1},\ \mymathbb{s}_{1j} \mapsto \mymathbb{u}_{2j},\ \mymathbb{s}_{n+1,j} \mapsto \mymathbb{u}_{3j},\ \mymathbb{s}_{n+2,j} \mapsto \mymathbb{u}_{4j}$.

    For types $E^{(1)}_{6,7,8}$, the isomorphism is provided by the simpler assignments
    \begin{equation}\label{E-typeiso}
    \begin{matrix}
        U_1 & \mapsto & (S_1\cdots S_{n+1}\cdots S_2)^{-1},\hfill\\
        U_2 & \mapsto & S_1,\hfill\\
        U_3 & \mapsto & \mathbf{S}\,S_{n+1}\mathbf{S}^{-1},\hfill\\
        T_i & \mapsto & S_{i+1},\hfill
    \end{matrix}
    \end{equation}
    where again $\mathbf{S}:=S_2\cdots S_n$. Indeed, for $d:=\mathrm{max}\{d_k\}\geq3,$ mirroring the proof of \Cref{lem:braidfourzero} one gets the braid theorem in form
    \begin{equation}
    \pi_1(\mathrm{UC}_n(\Sigma_{0,3})) \ \simeq \ \mathrm{Ar}([G(d,1,n)]_1),
    \end{equation}
    and the Hecke relations identify under the same specialization \eqref{specialization} and analogous identifications $\mymathbb{s}_{0j} \mapsto \mymathbb{u}_{1j}^{-1},\ \mymathbb{s}_{1j} \mapsto \mymathbb{u}_{2j},\ \mymathbb{s}_{n+1,j} \mapsto \mymathbb{u}_{3j}$.
\end{proof}

\begin{remark}
Under further specialization, for $n=1$ we recover the original definition of the GDAHA in rank one \cite[§5.1]{EOR2007}. E.g., adopting the notation as in loc. cit. for type $D^{(1)}_4$, it suffices to identify the generators as
\begin{equation}
    \begin{matrix}
        S_0 & \mapsto & qT_1^{-1},\hfill\\
        S_i & \mapsto & T_{i+1},\quad 1\leq i \leq 3,
    \end{matrix}
\end{equation}
and specialize all Hecke parameters:
\begin{equation}
    \begin{matrix}
        \mymathbb{s}_{01} \ \mapsto \ -qt_{11}, & \mymathbb{s}_{02} \ \mapsto \ qt_{11}^{-1},\hfill\\
        \mymathbb{s}_{11} \ \mapsto \ t_{21},\hfill & \mymathbb{s}_{12} \ \mapsto \ -t_{21}^{-1},\\
        \mymathbb{s}_{21} \ \mapsto \ t_{31},\hfill & \mymathbb{s}_{22} \ \mapsto \ -t_{31}^{-1},\\
        \mymathbb{s}_{31} \ \mapsto \ t_{41},\hfill & \mymathbb{s}_{32} \ \mapsto \ -t_{41}^{-1}.\\
    \end{matrix}
    \end{equation}
In particular, notice that adopting generic Hecke relations \eqref{generic-Hecke} maximizes the number of essential parameters, allowing to absorb $q$ through scaling.
\end{remark}
\begin{remark}
    In hindsight, \Cref{thm:GDAHA} dutifully proves the observation in \cite[Example 6.2]{EM2008}, clarifying all details with a transparent purely algebraic viewpoint.
\end{remark}
\begin{remark}
The remaining groups $[G(4,1,n)]_{2}$, $[G(4,2,n)]_{1,2}$, and $[G(6,2,n)]$, $[G(6,3,n)]$, are subgroups of $[G(4,1,n)]_{1}$ and $[G(6,1,n)]_{1}$, respectively \cite[Lemmas 3.2 and 3.3]{Malle1996}, by restricting some of the generators to their conjugates and/or squares.
On the contrary, such restriction---which hinges on the order relations---fails to capture the subgroups' generic Hecke algebras as (Hecke) subalgebras. We postpone a detailed study of these missing algebras to future work, as we hunt for a conceptual (deformation) theory encompassing crystallographic complex reflection groups as a whole (including non-Steinberg and, possibly, sporadic). 
\end{remark}

\appendix

\begin{landscape}
\section{Non-genuine Steinberg families}\label{app1}
\begin{table}[!h]
\begin{tabular}{ccccccc}\toprule
$\hspace{-.5em}W$ & & $\pazocal{D}$ & & $x$-relation & & extra order relation\\
\midrule\\
$[W(A_1)]^\alpha$ & & \raisebox{-2.75em}{\includegraphics[height=6em]{gfx/1A1.pdf}} & & none\!\!\! & & $(\sg{1}\sg{2}\sg{3})^2$ \\[3em]
$[W(A_{n-1})]^\alpha$, $n\geq3$ & & \raisebox{-2.75em}{\includegraphics[height=6em]{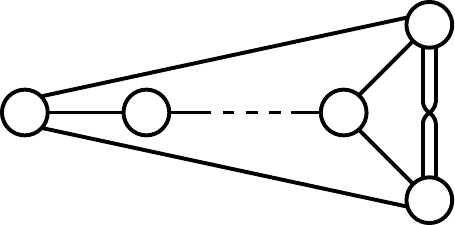}} & & $\sg{n+1}\sg{n-1}\cdots\sg{1}\sg{n}\isg{1} =\sg{n-1}\isg{n}\isg{n-1}\cdots\isg{1}\isg{n+1}$ & & $(\sg{1}\cdots\sg{n+1}\sg{n-1}\cdots\sg{2})^2$ \\[3.5em]
\midrule\\
$[G(2,1,n)]_1^\alpha$, $n\geq2$ & & \raisebox{-2.75em}{\includegraphics[height=6em]{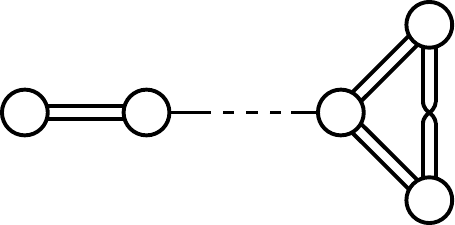}} & & $\sg{n}\sg{n+1}\sg{n}^{-1}\sg{n+2}=\sg{n+2}\sg{n}\sg{n+1}\sg{n}^{-1}$ & & $(\sg{1}\cdots\sg{n+2}\sg{n}\cdots\sg{2})^2$\\[3.5em]\bottomrule
\end{tabular}
\caption{The following table contains the Coxeter-like diagram $\pazocal{D}$, together with both the $x$-relation and the extra order relation, for the two (irreducible) infinite families of non-genuine Steinberg crystallographic complex reflection groups. Reminding that $[W(A_1)]^\alpha\simeq[G(2,1,1)]_1^\alpha$, the latter group is omitted so that we restrict to $n\geq2$ in type $C$. For this $1$-dimensional group only, $\pazocal{D}$ is a Coxeter diagram and no $x$-relation is involved. Nodes are increasingly numbered from left to right, and bottom to top when vertically aligned. }\label{tab:nongenuine}
\end{table}
\end{landscape}

\begin{section}{Genuine Steinberg families}\label{app2}
\vspace{-1em}
\begin{longtable}{ccccc}
\caption{The following table contains the Coxeter-like diagram $\pazocal{D}$, with its complementary extra order relation, for all (irreducible) infinite families of genuine Steinberg crystallographic complex reflection groups---as originally introduced in \cite{Malle1996,Puente2017}.
Nodes are increasingly numbered ``like a seismogram'': always from left to right, and bottom to top when vertically aligned.}\vspace{1em}
\label{tab:genuine}\\
\toprule
$W$ & & $\pazocal{D}$ & & extra order relation
\\\midrule\\
$[G(3,1,1)]$ & & \raisebox{-.25em}{\includegraphics[height=1.25em]{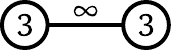}} & & $(\sg{1}\sg{2})^3$ \\[1.25em]
$[G(3,1,n)]_1$, $n\geq2$ & &  \raisebox{-.25em}{\includegraphics[height=1.25em]{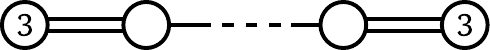}} & & $(\sg{1}\cdots\sg{n+1}\cdots\sg{2})^3$ \\[1em]
\midrule\\
$[G(4,1,1)]$ & & \raisebox{-.25em}{\includegraphics[height=1.25em]{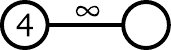}} & & $(\sg{1}\sg{2})^4$ \\[1.25em]
$[G(4,1,n)]_1$, $n\geq2$ & &  \raisebox{-.25em}{\includegraphics[height=1.25em]{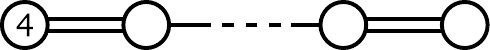}} & & $(\sg{1}\cdots\sg{n+1}\cdots\sg{2})^4$ \\[1em]
\midrule\\
$[G(4,1,2)]_2$ & & \raisebox{-.25em}{\includegraphics[height=1.25em]{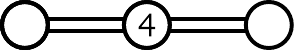}} & & $(\sg{1}\sg{2}\sg{3})^4$ \\[1.25em]
$[G(4,1,n)]_2$, $n\geq3$ & &  \raisebox{-.25em}{\includegraphics[height=3.75em]{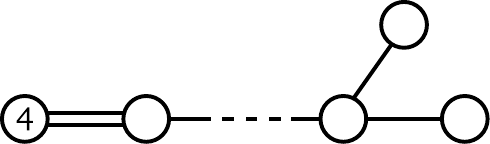}} & & $(\sg{1}\cdots\sg{n+1}\sg{n-1}\cdots\sg{2})^4$ \\[1em]
\midrule\\
$[G(4,2,2)]_1$ & & \raisebox{-2.75em}{\includegraphics[height=6em]{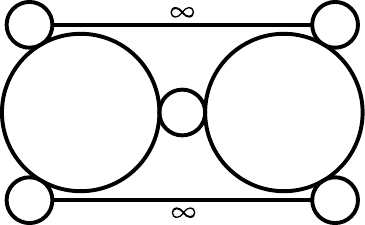}} & & $(\sg{2}\sg{1}\sg{3}\sg{4}\sg{5})^2$ \\[4em]
$[G(4,2,3)]_1$ & &  \raisebox{-2.75em}{\includegraphics[height=6em]{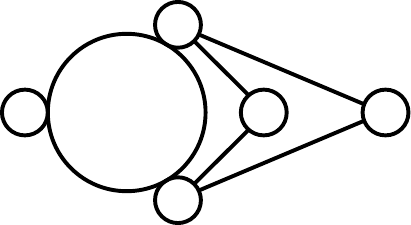}} & & $(\sg{1}\sg{2}\sg{3}\sg{4}\sg{5})^4$ \\[4em]
$[G(4,2,n)]_1$, $n\geq4$ & &  \raisebox{-2.75em}{\includegraphics[height=6em]{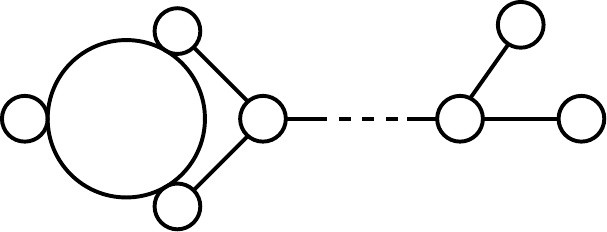}} & & $(\sg{1}\cdots\sg{n+2}\sg{n}\cdots\sg{4})^4$ \\[3.5em]
\midrule\\
$[G(4,2,2)]_2$ & & \raisebox{-2.75em}{\includegraphics[height=6em]{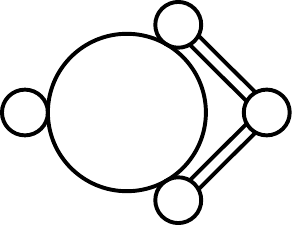}} & & $(\sg{1}\sg{2}\sg{3}\sg{4})^4$ \\[4em]
$[G(4,2,n)]_2$, $n\geq4$ & &  \raisebox{-2.75em}{\includegraphics[height=6em]{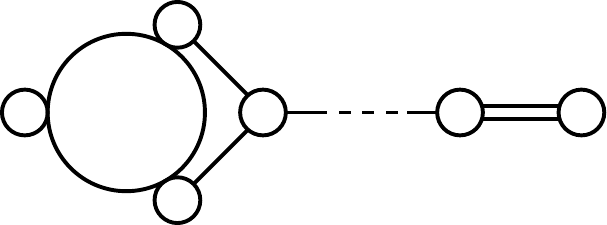}} & & $(\sg{1}\cdots\sg{n+2}\cdots\sg{4})^4$ \\[3.5em]
\midrule\\
$[G(6,1,1)]$ & & \raisebox{-.25em}{\includegraphics[height=1.25em]{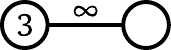}} & & $(\sg{1}\sg{2})^6$ \\[1.25em]
$[G(6,1,n)]$, $n\geq2$ & &  \raisebox{-.25em}{\includegraphics[height=1.25em]{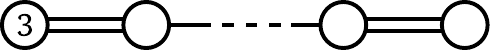}} & & $(\sg{1}\cdots\sg{n+1}\cdots\sg{2})^6$ \\[1em]
\midrule\\
$[G(6,2,2)]$ & & \raisebox{-.25em}{\includegraphics[height=1.25em]{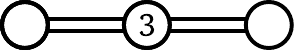}} & & $(\sg{1}\sg{2}\sg{3})^6$ \\[.75em]
$[G(6,2,n)]$, $n\geq3$ & &  \raisebox{-.25em}{\includegraphics[height=3.75em]{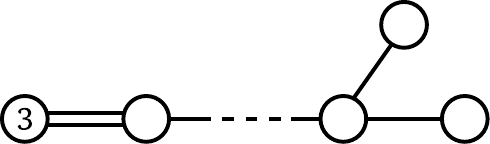}} & & $(\sg{1}\cdots\sg{n+1}\sg{n-1}\cdots\sg{2})^6$ \\[1em]
\midrule\\
$[G(6,3,2)]$ & & \raisebox{-2.75em}{\includegraphics[height=6em]{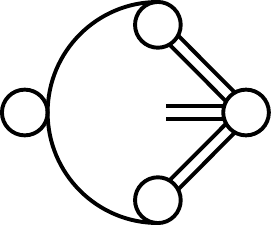}} & & $(\sg{2}\sg{3}\sg{4})^6$ \\[4em]
$[G(6,3,n)]$, $n\geq2$ & &  \raisebox{-2.75em}{\includegraphics[height=6em]{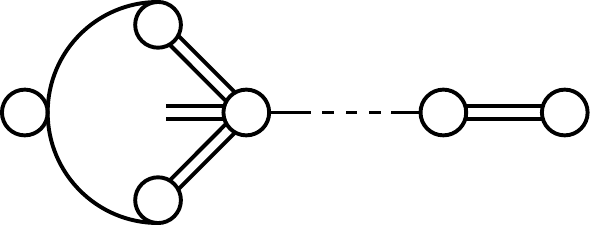}} & & $(\sg{2}\cdots\sg{n+2}\cdots\sg{4})^6$
\\[3.5em]\bottomrule
\end{longtable}
\end{section}

{\small
\bibliographystyle{abbrv}
\bibliography{references}

\begin{thebibliography}{10}

\bibitem{BS1978}
J.~N. Bernstein and O.~W. Schwarzman.
\newblock { Chevalley’s theorem for complex crystallographic Coxeter groups [Russian]}.
\newblock {\em {Funktsional'nyi Analiz i ego Prilozheniya}}, 12, 1978.

\bibitem{BMR1998}
M.~Broué, G.~Malle, and R.~Rouquier.
\newblock {Complex reflection groups, braid groups, Hecke algebras}.
\newblock {\em Journal für die reine und angewandte Mathematik}, 500, 1998.

\bibitem{Cohen1976}
A.~M. Cohen.
\newblock Finite complex reflection groups.
\newblock {\em Annales scientifiques de l'École Normale Supérieure}, 9, 1976.

\bibitem{Dubrovin1996}
B.~Dubrovin.
\newblock {Geometry of 2D topological field theories}.
\newblock {\em Integrable Systems and Quantum Groups. Lecture Notes in Mathematics}, 1620, 1996.

\bibitem{EM2008}
P.~Etingof, W.~L. Gan, and X.~Ma.
\newblock {On Elliptic Dunkl Operators}.
\newblock {\em Michigan Mathematical Journal}, 57, 2008.

\bibitem{EGO2006}
P.~Etingof, W.~L. Gan, and A.~Oblomkov.
\newblock {G}eneralized double affine {H}ecke algebras of higher rank.
\newblock {\em Journal fu\"r die reine und angewandte Mathematik}, 600, 2006.

\bibitem{EOR2007}
P.~Etingof, A.~Oblomkov, and E.~Rains.
\newblock {G}eneralized double affine {H}ecke algebras of rank $1$ and quantized {D}el {P}ezzo surfaces.
\newblock {\em Advances in Mathematics}, 212, 2007.

\bibitem{GK2010}
V.~Goryunov and D.~Kerner.
\newblock {Automorphisms of $P_8$ singularities and the complex crystallographic groups}.
\newblock {\em Proceedings of the Steklov Institute of Mathematics}, 267, 2009.

\bibitem{IR2025}
B.~Ion and E.~Roller.
\newblock {The braid skein algebra of the Heisenberg manifold}.
\newblock {\em \href{https://sites.pitt.edu/~bion/pdf/heisenberg-skein-v2.pdf}{sites.pitt.edu/bion}}, 2025.

\bibitem{IS2020}
B.~Ion and S.~Sahi.
\newblock {\em Double affine Hecke algebras and congruence groups}, volume 268.
\newblock Memoirs of the American Mathematical Society, 2020.

\bibitem{KTY1982}
J.~Kaneko, S.~Tokunaga, and M.~Yoshida.
\newblock {Complex crystallographic groups II}.
\newblock {\em {Journal of the Mathematical Society of Japan}}, 34, 1982.

\bibitem{Lambropoulou2000}
S.~Lambropoulou.
\newblock {Braid structures in knot complements, handlebodies and $3$–manifolds}.
\newblock {\em Knots in Hellas '98, Series on Knots and Everything}, 2000.

\bibitem{Malle1996}
G.~Malle.
\newblock {Presentations for crystallographic complex reflection groups}.
\newblock {\em Transformation Groups}, 1, 1996.

\bibitem{Popov2022}
V.~L. Popov.
\newblock Discrete complex reflection groups.
\newblock {\em {Communications in Mathematics}}, 30, 2022.

\bibitem{Puente2017}
P.~C. Puente.
\newblock {Crystallographic complex reflection groups and the braid conjecture (PhD dissertation)}.
\newblock {\em \href{https://doi.org/10.12794/metadc1011877}{University of North Texas digital library}}, 2017.

\bibitem{PS2019}
P.~C. Puente and A.~V. Shepler.
\newblock Steinberg's theorem for crystallographic complex reflection groups.
\newblock {\em Journal of Algebra}, 522, 2019.

\bibitem{Sahi1999}
S.~Sahi.
\newblock {Nonsymmetric Koornwinder Polynomials and Duality}.
\newblock {\em Annals of Mathematics}, 150, 1999.

\bibitem{ST1997}
K.~Saito and T.~Takebayashi.
\newblock {Extended Affine Root Systems III (Elliptic Weyl Groups)}.
\newblock {\em Publications of the Research Institute for Mathematical Sciences}, 33, 1999.

\bibitem{SS2009}
S.~Saito and M.~Shiota.
\newblock {On Hecke algebras associated with elliptic root systems and the double affine Hecke algebras}.
\newblock {\em Publications of the Research Institute for Mathematical Sciences}, 45, 2009.

\bibitem{Tak1994}
T.~Takebayashi.
\newblock {Defining Relations of Weyl Groups for Extended Affine Root Systems $A^{(1,1)}_l$, $B^{(1,1)}_l$, $C^{(1,1)}_l$, $D^{(1,1)}_l$}.
\newblock {\em Journal of Algebra}, 168, 1994.

\bibitem{TY1982}
S.~Tokunaga and M.~Yoshida.
\newblock {Complex crystallographic groups I}.
\newblock {\em {Journal of the Mathematical Society of Japan}}, 34, 1982.

\bibitem{vanderLek1983}
H.~van~der Lek.
\newblock {The homotopy type of complex hyperplane complements (PhD dissertation)}.
\newblock {\em \href{https://hdl.handle.net/2066/148301}{Radboud University digital library}}, 1983.

\end{thebibliography}
}

\end{document}